\documentclass[10pt]{article}

\usepackage{latexsym, amssymb, amsmath, amsthm, amscd}
\usepackage[all,cmtip]{xy}
\usepackage{amsmath}
\usepackage{amsthm}
\usepackage{amssymb}
\usepackage{amsfonts}
\usepackage{amsrefs}
\usepackage{color,authblk}
\usepackage{tikz}
\usepackage{amscd} 
\usepackage{comment}
\usepackage{mathrsfs}
\usepackage{comment}
\usepackage[all]{xy}
\usepackage{graphicx}
\usepackage{authblk}
\usepackage{hyperref}
\usepackage{fullpage}
\usepackage[all,cmtip]{xy}
\usepackage{relsize}
\usepackage{amsmath,amscd}
\usepackage{tikz-cd}
\usepackage{tikz}
\usetikzlibrary{matrix}
\usepackage[mathcal]{euscript}
\usepackage{txfonts}

\usepackage{hyperref}

\numberwithin{equation}{section} \DeclareMathSizes{2}{10}{12}{13}
\makeatletter
\newcommand*{\doublerightarrow}[2]{\mathrel{
  \settowidth{\@tempdima}{$\scriptstyle#1$}
  \settowidth{\@tempdimb}{$\scriptstyle#2$}
  \ifdim\@tempdimb>\@tempdima \@tempdima=\@tempdimb\fi
  \mathop{\vcenter{
    \offinterlineskip\ialign{\hbox to\dimexpr\@tempdima+1em{##}\cr
    \rightarrowfill\cr\noalign{\kern.5ex}
    \rightarrowfill\cr}}}\limits^{\!#1}_{\!#2}}}
\newcommand*{\triplerightarrow}[1]{\mathrel{
  \settowidth{\@tempdima}{$\scriptstyle#1$}
  \mathop{\vcenter{
    \offinterlineskip\ialign{\hbox to\dimexpr\@tempdima+1em{##}\cr
    \rightarrowfill\cr\noalign{\kern.5ex}
    \rightarrowfill\cr\noalign{\kern.5ex}
    \rightarrowfill\cr}}}\limits^{\!#1}}}
\makeatother

\newtheorem{thm}{Proposition}[section]
\newtheorem{Thm}[thm]{Theorem}

\newtheorem{cor}[thm]{Corollary}

\newtheorem{lem}[thm]{Lemma}
\newtheorem{defn}[thm]{Definition}
\newtheorem{prop}[thm]{Proposition}

\title{Comodule theories in Grothendieck categories and relative Hopf objects}

\author{Mamta Balodi \footnote{Stat-Math Unit, Indian Statistical Institute, Bangalore, India. Email: mamta.balodi@gmail.com.} $\qquad$ Abhishek Banerjee \footnote{Department of Mathematics, Indian Institute of Science, Bangalore, India. Email: abhishekbanerjee1313@gmail.com.}
$\qquad$ Surjeet Kour \footnote{Department of Mathematics, Indian Institute of Technology, Delhi, India. Email: koursurjeet@gmail.com.}}

\date{}

\begin{document}

\maketitle 

\medskip

\begin{abstract} We develop the categorical algebra of the noncommutative base change of a comodule category by means of a Grothendieck category $\mathfrak S$. We describe when the resulting category of comodules is locally finitely generated, locally noetherian or may be recovered as a coreflective subcategory of the noncommutative base change of a module category. We also introduce the category ${_A}\mathfrak S^H$ of relative $(A,H)$-Hopf modules in $\mathfrak S$, where $H$ is a Hopf algebra and $A$ is a right $H$-comodule algebra. We study the cohomological theory in ${_A}\mathfrak S^H$ by means of spectral sequences. Using coinduction functors and functors of coinvariants, we study torsion theories and how they relate to injective resolutions in ${_A}\mathfrak S^H$. Finally, we use the theory of associated primes and support in noncommutative base change of module categories to give direct sum decompositions of minimal injective resolutions in ${_A}\mathfrak S^H$. 
\end{abstract}

\smallskip

{\bf MSC(2020) Subject Classification:} 16T05, 16T15, 18E10, 18E40


\smallskip

\section{Introduction}

In this paper we aim to study a theory of modules over a noncommutative base change of categorified fiber bundles. We consider a Hopf algebra $H$ over a field $k$, and a $k$-algebra $A$ which carries the structure of a right $H$-comodule algebra. The geometric version of this consists of an affine algebraic group scheme $G$ acting freely on an affine scheme $X$. Accordingly, the role of the quotient space $X/G$ is played by the algebra $B=A^{coH}$ of $H$-coinvariants of $A$. The sheaves over the quotient space are then replaced by ``relative $(A,H)$-Hopf modules'' (see, for  instance, \cite{Ca98}, \cite{Doi83}, \cite{Doi92}, \cite{Tak79}). This consists of a left $A$-module $M$ equipped with a right $H$-comodule structure $M\longrightarrow M\otimes H$ such that the multiplication
$A\otimes M\longrightarrow M$ is compatible with the $H$-coactions on $A$ and $M$. Under certain faithful flatness conditions (see Schneider \cite{Sch}), the category of relative $(A,H)$-Hopf modules becomes identical to the category of modules over $B=A^{coH}$. 

\smallskip
The idea of noncommutative base change of an algebra by means of a Grothendieck category goes back to Popescu \cite{Pop}. If $R$ is a $k$-algebra, and $\mathfrak S$ is a $k$-linear Grothendieck category, an $R$-module object in $\mathfrak S$ consists of an object $\mathcal M\in \mathfrak S$ along with a morphism $\rho_{\mathcal M}:R\longrightarrow \mathfrak S(\mathcal M,\mathcal M)$ of $k$-algebras. If $\mathfrak S$ is the category of modules over a $k$-algebra $R'$, then an $R$-module object in $\mathfrak S$ is the same as a module over
$R\otimes_kR'$. This justifies the expression ``noncommutative base change'' for the category $\mathfrak S_R$ of (right) $R$-module objects in $\mathfrak S$. In \cite{AZ}, Artin and Zhang have developed in great detail the categorical aspects of this theory, studying finitely generated objects and finitely presented objects in $\mathfrak S_R$, as well as the properties of being locally finitely generated or locally noetherian. When applied to a certain quotient of the category of graded modules over a graded algebra, they capture the geometry of the noncommutative projective schemes of Artin and Zhang \cite{AZ0x} (see also
Artin, Small and Zhang \cite{AZ1x}). The authors in \cite{AZ} also developed the categorical algebra and the  homological algebra of $R$-module objects in $\mathfrak S$, by establishing versions of localization, Nakayama Lemma, Hilbert Basis Theorem and the derived functors $Tor$ and $Ext$. Other aspects such as coherence properties and deformation theory in $\mathfrak S_R$ have been studied by Lowen and Van den Bergh \cite{LV}. The non-noetherian descent in these module categories with respect to the classical Beauville-Laszlo theorem (see 
\cite{BL1}, \cite{BL2})
has been studied in \cite{ABBs}. Our starting point in this paper is the notion of  a $C$-comodule object in a Grothendieck category $\mathfrak S$, where $C$ is a $k$-coalgebra, introduced by Brzezi\'{n}ski and Wisbauer \cite[$\S$ 39]{Wis}.

\smallskip
This paper consists of two parts. In the first part, we work with the categorical properties of $\mathfrak S^C$, the collection of (right) $C$-comodule objects in $\mathfrak S$. We show that $\mathfrak S^C$ may be recovered as a coreflective subcategory of ${_A}\mathfrak S$, where $C\otimes A\longrightarrow k$ is an $\mathfrak S$-rational linear pairing (see
Definition \ref{D31v}) of $C$ with a $k$-algebra $A$. We obtain a family of generators for $\mathfrak S^C$, and show that  $\mathfrak S^C$ is a Grothendieck category. When $\mathfrak S$ is locally finitely generated, we use $\mathfrak S$-rational pairings to give conditions for $\mathfrak S^C$ to be locally finitely generated.  We also study injective envelopes in $\mathfrak S^C$, and give conditions for $\mathfrak S^C$ to be locally noetherian.

\smallskip In the second part of the paper, we study the category ${_A}\mathfrak S^H$ of relative $(A,H)$-Hopf module objects in $\mathfrak S$, where $A$ is a right $H$-comodule algebra. As indicated before, the objects in ${_A}\mathfrak S^H$ may be seen as modules over the ``quotient space of ${_A}\mathfrak S$'' with respect to the base change of the $H$-coaction on $A$. We will study the cohomology theory and the properties of relative  $(A,H)$-Hopf module objects in $\mathfrak S$ by adapting the methods of Caenepeel and Gu\'{e}d\'{e}non \cite{CG}. First, we construct a spectral sequence that computes the higher derived Hom functors in ${_A}\mathfrak S^H$ (see Theorem \ref{T6.3spec})
\begin{equation}
E_2^{pq}=R^p(-)^{co H}(R^q{_{{_A}\mathfrak S}}HOM(\mathcal M,\_\_)(\mathcal N)) \Rightarrow (R^{p+q}{_A}\mathfrak S^H(\mathcal M,\_\_))(\mathcal N)
\end{equation}  for $\mathcal M$, $\mathcal N\in {_A}\mathfrak S^H$. Here, $(-)^{coH}$ is the functor of taking $H$-coinvariants on the category $Com-H$ of right $H$-comodules. The functor ${_{{_A}\mathfrak S}}HOM(\mathcal M,\_\_)$ for $\mathcal M\in {_A}\mathfrak S^H$ is the rational Hom that is right adjoint to $\mathcal M\otimes\_\_: Com-H\longrightarrow {_A}\mathfrak S^H$. We show that the rational Hom object ${_{{_A}\mathfrak S}}HOM(\mathcal M,\mathcal N)$ may also be recovered as the largest $H$-comodule contained inside the morphism space ${_A}\mathfrak S(\mathcal M,\mathcal N)$. 

\smallskip Thereafter, we suppose that the Hopf algebra $H$ is cosemisimple and 
let $B=A^{coH}$ be the subalgebra of coinvariants of $A$. We obtain an adjoint triple of functors
\begin{equation}\label{triple4c}
A\otimes_B\_\_:{_B}\mathfrak S\longrightarrow {_A}\mathfrak S^H\qquad \mathfrak C^H:{_A}\mathfrak S^H\longrightarrow {_B}\mathfrak S \qquad {_B}HOM(A,\_\_):{_B}\mathfrak S\longrightarrow {_A}\mathfrak S^H
\end{equation}
relating ${_A}\mathfrak S^H$ to the category ${_B}\mathfrak S$ of $B$-module objects in $\mathfrak S$.  Here $\mathfrak C^H$ is   the functor of $H$-coinvariants, and we construct the coinduction functor $ {_B}HOM(A,\_\_):{_B}\mathfrak S\longrightarrow {_A}\mathfrak S^H$. We use the functors in \eqref{triple4c} to build torsion theories in ${_A}\mathfrak S^H$ as well as  to relate injective envelopes and minimal injective resolutions in  ${_A}\mathfrak S^H$ and  ${_B}\mathfrak S$. If $\mathcal N\in {_B}\mathfrak S$, we show that $\mathfrak C^H({_A}\mathcal E^H({_B}HOM(A,\mathcal N)))\cong {_B}\mathcal E(\mathcal N)$, where ${_A}\mathcal E^H(-)$ denotes the injective envelope in ${_A}\mathfrak S^H$ and 
${_B}\mathcal E(-)$ denotes the injective envelope in ${_B}\mathfrak S$. We also show that the full subcategory \begin{equation}\label{trix1}
 \mathcal T({_A}\mathfrak S^H):=\{\mbox{$\mathcal M\in {_A}\mathfrak S^H$ $\vert$ $\mathfrak C^H(\mathcal M)=0$ }\} 
\end{equation} is a torsion class in ${_A}\mathfrak S^H$. The corresponding torsion free class $ \mathcal F({_A}\mathfrak S^H)$ consists of those $\mathcal M\in {_A}\mathfrak S^H$ for which the canonical morphism
$\mathcal M\longrightarrow {_B}HOM(A,\mathfrak C^H(\mathcal M)) $ is a monomorphism. In terms of injective cogenerators of $\mathfrak S$, we give conditions for all objects in
${_A}\mathfrak S^H$ to be torsion free, i.e., $ \mathcal F({_A}\mathfrak S^H)={_A}\mathfrak S^H$.

\smallskip Finally, we apply some of our results from \cite{BanKr} on associated primes and support in abstract module categories. If   $\mathfrak S$ is a strongly locally noetherian Grothendieck category and $T$ is a commutative and noetherian $k$-algebra, we have shown in \cite{BanKr} that any injective object in ${_T}\mathfrak S$ decomposes as a direct sum of injective envelopes of ``$T$-elementary objects in $\mathfrak S$''. In the abstract module category ${_T}\mathfrak S$, the  $T$-elementary objects  introduced in \cite{BanKr} play a role similar to quotients over prime ideals. In this paper, we show that any injective $\mathcal E\in {_T}\mathfrak S$ decomposes as a direct sum (see Proposition \ref{P8.18xk})
\begin{equation}
\mathcal E\cong \underset{\mathfrak p\in Ass_{{_T}\mathfrak S}(\mathcal E)}{\bigoplus}\textrm{ }\mathcal E(\mathfrak p)
\end{equation}
where $\mathcal E(\mathfrak p)\subseteq \mathcal E$    is a direct sum of injective envelopes of $T$-elementary objects
 and  maximal   with respect to the  collection of subobjects $\mathcal N\subseteq \mathcal E$ such that $Ass_{{_T}\mathfrak S}(\mathcal N)=\{\mathfrak p\}$. We refine some of our results from \cite{BanKr} and combine with the methods of   Caenepeel and Gu\'{e}d\'{e}non \cite{CG} to obtain a direct sum decomposition of objects appearing in a minimal injective resolution ${_A}\mathcal E^{H\bullet}(\mathcal M)$ 
of $\mathcal M\in {_A}\mathfrak S^H$ in terms of injective envelopes of elementary objects associated to prime ideals in $B=A^{coH}$ (see Theorem \ref{Tfin6}) \begin{equation}\label{827qzva}
{_A}\mathcal E^{H\bullet}(\mathcal M)= \underset{\mathfrak p\in Ass_{{_B}\mathfrak S}({_B}\mathcal E^\bullet(\mathfrak C^H(\mathcal M)))}{\bigoplus}\textrm{ } {_B}HOM(A,{_B}\mathcal E^\bullet(\mathfrak C^H(\mathcal M))(\mathfrak p))
\end{equation} Similar to \cite{CG}, this is done under certain additional noetherian and torsion freeness assumptions that we extend to the datum $(A,H,\mathfrak S)$. As such, we hope that this paper will be the first step towards a systematic development of the theory of $(A,H)$-Hopf modules and closely related notions in a Grothendieck category $\mathfrak S$. The literature on the usual $(A,H)$-Hopf modules, as well as Doi-Hopf modules, Yetter-Drinfeld modules and entwined modules is vast, for which we refer the reader, for instance, to  \cite{BBR1}, \cite{BBR2}, \cite{Ban}, \cite{Zhou}, \cite{Brz99},  \cite{BCMZ}, \cite{Brz2002}, \cite{CMZ}, \cite{CMZ0}, \cite{CMIZ}, \cite{Gued}.

\section{$\mathfrak S$-Rational pairings, module objects and comodule objects}

Let $k$ be a field. Throughout, we let $\mathfrak S$ be a $k$-linear Grothendieck category. If $R$ is a $k$-algebra, a (right) $R$-module object in $\mathfrak S$ (see Popescu \cite[p 108]{Pop}) is a pair $(\mathcal M,\rho_{\mathcal M})$, where $\mathcal M$ is an object of $\mathfrak S$ and $\rho_{\mathcal M}:R\longrightarrow  \mathfrak S(\mathcal M,\mathcal M)$ is a morphism of $k$-algebras, i.e., $\rho_{\mathcal M}(rs)=\rho_{\mathcal M}(s)\circ \rho_{\mathcal M}(r)$ for $r$, $s\in R$.  A morphism $\phi:(\mathcal M,\rho_{\mathcal M})\longrightarrow (\mathcal M',\rho_{\mathcal M'})$ of $R$-module objects consists of a morphism $\phi:\mathcal M\longrightarrow \mathcal M'$ in $\mathfrak S$ such that $\rho_{\mathcal M'}\circ \phi=\phi\circ \rho_{\mathcal M}$. The category of right $R$-module objects in $\mathfrak S$ will be denoted by $\mathfrak S_R$. We will usually write an object $(\mathcal M,\rho_{\mathcal M})\in \mathfrak S_R$ simply as $\mathcal M\in \mathfrak S_R$. By \cite[Proposition B2.2]{AZ}, we know that $\mathfrak S_R$ is also a Grothendieck category. We may similarly consider
the category ${_R}\mathfrak S$ of left $R$-module objects in $\mathfrak S$. 

\smallskip If $R-Mod$ denotes the category of left $R$-modules, there is a bifunctor
\begin{equation}\label{bif}
\_\_\otimes_R \_\_ : \mathfrak S_R \times R-Mod\longrightarrow \mathfrak S
\end{equation} For a detailed study of the properties of this bifunctor, we refer the reader to Artin and Zhang \cite{AZ}. For any $\mathcal M\in \mathfrak S_R$, we have $\mathcal M\otimes_RR\cong \mathcal M$. If $\mathcal M\in \mathfrak S_R$ and $\mathcal N\in \mathfrak S$, we note that there is a canonical left $R$-module structure on $\mathfrak S(\mathcal M,\mathcal N)$. The functor $\mathcal M\otimes_R\_\_:  R-Mod\longrightarrow \mathfrak S$ is left adjoint to $\mathfrak S(\mathcal M,\_\_):\mathfrak S\longrightarrow R-Mod$ (see \cite[$\S$ B3]{AZ}). In addition, if $\mathcal M\otimes_R\_\_:  R-Mod\longrightarrow \mathfrak S$  is exact, $\mathcal M\in \mathfrak S_R$ is said to be $R$-flat. 

\smallskip
On the other hand, if $V$ is a left $R$-module, the functor $\_\_\otimes_RV:\mathfrak S_R\longrightarrow \mathfrak S$ is also right exact and has a right adjoint (see \cite[Proposition B3.6]{AZ}) which we denote by $\underline{Hom}(V,\_\_):\mathfrak S\longrightarrow \mathfrak S_R$. In addition, if $\_\_\otimes_RV:\mathfrak S_R\longrightarrow \mathfrak S$ is exact, then $V\in R-Mod$ is said to be $\mathfrak S_R$-flat. Also, if $V$ is an $(R,R')$-bimodule, the functor $\_\_\otimes_RV$ takes values in $\mathfrak S_{R'}$. 

\smallskip
In particular, we always write $\otimes:=\otimes_k$. A right $R$-module object in $\mathfrak S$ may also be described in the following manner (see \cite[Lemma B3.14]{AZ}). 

\begin{defn}\label{D2.1}
An object $(\mathcal M,\rho_{\mathcal M})\in \mathfrak S_R$ is determined by a morphism $\mu_{\mathcal M}:\mathcal M\otimes R\longrightarrow \mathcal M$ in $\mathfrak S=\mathfrak S_k$ satisfying the following two conditions

\smallskip
(1) If $u:k\longrightarrow R$ is the unit of $R$, the composition $\mathcal M=\mathcal M\otimes k\xrightarrow{\mathcal M\otimes u}\mathcal M\otimes R\xrightarrow{\mu_{\mathcal M}}\mathcal M$ is the identity.

\smallskip
(2) The following diagram commutes
\begin{equation}
\begin{CD}
\mathcal M\otimes R\otimes R @>\mathcal M\otimes \mu_R>> \mathcal M\otimes R\\
@V\mu_{\mathcal M}\otimes RVV @VV\mu_{\mathcal M}V\\
\mathcal M\otimes R @>\mu_{\mathcal M}>> \mathcal M\\
\end{CD}
\end{equation} where $\mu_R$ is the multiplication on $R$. 
We will denote by $\mu'_{\mathcal M}:\mathcal M\longrightarrow \underline{Hom}(R,\mathcal M)$ the morphism corresponding to $\mu_{\mathcal M}:\mathcal M\otimes R\longrightarrow \mathcal M$ via the adjunction. 
\end{defn}

Similarly, there is the notion of $C$-comodule objects in $\mathfrak S$, where $C$ is a $k$-coalgebra (see \cite[$\S$ 39]{Wis}).

\begin{defn}\label{D2.2} Let $C$ be a $k$-coalgebra. A right $C$-comodule object in $\mathfrak S$ consists of an object $\mathcal P\in \mathfrak S$ and a morphism
$\Delta_{\mathcal P}: \mathcal P\longrightarrow \mathcal P\otimes C$ satisfying the following two conditions

\smallskip
(1) If $\epsilon: C\longrightarrow k$ is the counit of $C$, the composition $\mathcal P\xrightarrow{\Delta_{\mathcal P}}\mathcal P\otimes C\xrightarrow{\mathcal P\otimes \epsilon}\mathcal 
P$ is the identity.

\smallskip
(2) The following diagram commutes
\begin{equation}
\begin{CD}
\mathcal P  @>\Delta_{\mathcal P}>> \mathcal P\otimes C\\
@V\Delta_{\mathcal P}VV @VV{\mathcal P\otimes \Delta_C} V\\
\mathcal P\otimes C @>\Delta_{\mathcal P}\otimes C>> \mathcal P\otimes C\otimes C\\
\end{CD}
\end{equation} where $\Delta_C$ is the comultiplication on $C$. 
A morphism $\phi:(\mathcal P,\Delta_{\mathcal P})\longrightarrow (\mathcal P',\Delta_{\mathcal P'})$ of right $C$-comodules in $\mathfrak S$ consists of a morphism $\phi: \mathcal P
\longrightarrow \mathcal P'$ in $\mathfrak S$ such that $\Delta_{\mathcal M'}\circ \phi=(\phi\otimes C)\circ \Delta_{\mathcal M}$. The category of right $C$-comodule objects in $\mathfrak S$ will be denoted by $\mathfrak S^C$. Similarly, there is a category of left $C$-comodule objects in $\mathfrak S$, which will be denoted by $ {^C}\mathfrak S$. 
\end{defn}

Since $k$ is a field, the coalgebra $C$ is always flat over $k$. Using \cite[Proposition C1.7]{AZ}, it now follows that $\_\_\otimes C$ is exact. In fact, since every monomorphism in $k-Mod$ splits, we see that the bifunctor $\_\_\otimes \_\_: \mathfrak S\times k-Mod\longrightarrow \mathfrak S$ is exact in both variables. Accordingly, we see that $\mathfrak S^C$ contains both kernels and cokernels. This makes $\mathfrak S^C$ an abelian category. Further,  colimits and finite limits of systems in $\mathfrak S^C$ may be computed in $\mathfrak S$. 

\smallskip
We also observe that if $(\mathcal P,\Delta_{\mathcal P})$, $(\mathcal P',\Delta_{\mathcal P'})$ are $C$-comodule objects in $\mathfrak S$, then $\mathfrak S^C(\mathcal P,\mathcal P')$ can be expressed as the equalizer
\begin{equation}\label{comeq}
\mathfrak S^C(\mathcal P,\mathcal P')=Eq\left(\mathfrak S(\mathcal P,\mathcal P')\doublerightarrow{\qquad \phi\mapsto \Delta_{\mathcal P'}\circ \phi\qquad }{\phi\mapsto (\phi\otimes C)\circ \Delta_{\mathcal P}}\mathfrak S(\mathcal P,\mathcal P'\otimes C)\right)
\end{equation} Now let $f:C\longrightarrow D$ a morphism of $k$-coalgebras. Then, there is an obvious functor $f^*:\mathfrak S^C\longrightarrow 
\mathfrak S^D$ that takes $(\mathcal P,\Delta_{\mathcal P})$ to $(\mathcal P,(\mathcal P\otimes f)\circ \Delta_{\mathcal P})$. On the other hand, if
$(\mathcal Q,\Delta_{\mathcal Q})\in \mathfrak S^D$, we set
\begin{equation}\label{cot2}
\mathcal Q\square_DC :=Eq\left(\mathcal Q\otimes C\doublerightarrow{\qquad \Delta_{\mathcal Q}\otimes C\qquad }{(\mathcal Q\otimes f\otimes C)\circ (\mathcal Q\otimes \Delta_C)}\mathcal Q\otimes D\otimes C\right)
\end{equation} This defines a functor $f_*:\mathfrak S^D\longrightarrow \mathfrak S^C$.

\begin{thm}\label{P2.3k}
Let $f:C\longrightarrow D$ be a morphism of $k$-coalgebras.  Then, $(f^*,f_*)$ is a pair of adjoint functors.
\end{thm}

\begin{proof}
We consider $(\mathcal P,\Delta_{\mathcal P})\in \mathfrak S^C$ and $(\mathcal Q,\Delta_{\mathcal Q})\in \mathfrak S^D$. We consider $\phi\in \mathfrak S^D(f^*\mathcal P,\mathcal Q)$, i.e., $\phi\in \mathfrak S(\mathcal P,\mathcal Q)$ such that $ \Delta_{\mathcal Q}\circ \phi=(\phi\otimes D)\circ \Delta_{f^*(\mathcal P)}=(\phi\otimes D)\circ (\mathcal P\otimes f)\circ \Delta_{\mathcal P}$. Then, we have
\begin{equation}
\begin{array}{ll}
(\Delta_{\mathcal Q}\otimes C)\circ (\phi\otimes C)\circ \Delta_{\mathcal P}&=(\phi\otimes D\otimes C)\circ (\mathcal P\otimes f\otimes C)\circ (\Delta_{\mathcal P}\otimes C)\circ \Delta_{\mathcal P}\\
&=(\mathcal Q\otimes f \otimes C)\circ (\phi\otimes C\otimes C)\circ (\mathcal P\otimes \Delta_C)\circ \Delta_{\mathcal P}\\
&=(\mathcal Q\otimes f \otimes C)\circ (\mathcal Q\otimes \Delta_C)\circ (\phi\otimes C)\circ \Delta_{\mathcal P}\\
\end{array}
\end{equation} From the definition in \eqref{cot2}, it follows that $(\phi\otimes C)\circ \Delta_{\mathcal P}$ factors through $\mathcal Q\square_DC$ and it may be verified that this gives a morphism in $\mathfrak S^C(\mathcal P,\mathcal Q\square_DC)$. Conversely, if $\psi\in \mathfrak S^C(\mathcal P,\mathcal Q\square_DC)$, it may be verified that the composition $\mathcal P\xrightarrow{\psi}\mathcal Q\square_DC\hookrightarrow \mathcal Q\otimes C
\xrightarrow{\mathcal Q\otimes \epsilon_C}\mathcal Q$ is a morphism in $S^D(f^*\mathcal P,\mathcal Q)$, where $\epsilon_C:C\longrightarrow k$ is the counit. This proves the result.
\end{proof} 

Applying Proposition \ref{P2.3k} to the counit morphism $\epsilon_C: C\longrightarrow k$, we have in particular that
\begin{equation}\label{adj2.7}
\mathfrak S(\mathcal P,\mathcal Q)=\mathfrak S^C(\mathcal P,\mathcal Q\otimes C)
\end{equation} for any $\mathcal P\in \mathfrak S^C$ and $\mathcal Q\in \mathfrak S$. 

\smallskip
We now let $C$ be a $k$-coalgebra and $A$ be a $k$-algebra. A pairing of $C$ and $A$ is a $k$-linear map $\beta:C \otimes A \longrightarrow k$ such that the induced map $\hat{\beta}:A \longrightarrow C^*$ is  morphism of $k$-algebras. Here, the multiplication on $C^*$ is defined by setting
\begin{equation}
(f \ast g)(c)=\sum f(c_1)g(c_2)
\end{equation} for $c\in C$ and $f$, $g\in C^\ast$. 
We now consider the category ${_A}\mathfrak S$ of left $A$-module objects in $\mathfrak S$ and the category $\mathfrak S^C$ of right $C$-comodule objects in $\mathfrak S$. If $\beta: C\otimes A\longrightarrow k$ is a pairing, then any $C$-comodule object $\mathcal M$ in $\mathfrak S$ is naturally an $A$-module object in $\mathfrak S$ by
\begin{equation}\label{str3d}
A\otimes \mathcal{M}  \xrightarrow{A\otimes \Delta_\mathcal{M}} A\otimes \mathcal{M} \otimes C \xrightarrow{\cong }\mathcal M\otimes C\otimes A \xrightarrow{\mathcal{M} \otimes \beta} \mathcal{M}
\end{equation}
Thus, we have an embedding ${_A}{\mathfrak I}^C:\mathfrak S^C \longrightarrow {_A}\mathfrak S$ of categories. We recall that the functor $\_\_\otimes A:\mathfrak S \longrightarrow \mathfrak S$ has the right adjoint $\underline{Hom}(A,-):\mathfrak S \longrightarrow \mathfrak S$. 
It also follows that for a finite dimensional vector space $V$ and an object $\mathcal{M} \in \mathfrak S$, we have 
\begin{equation}\label{fd}
\mathcal{M} \otimes V \cong \underline{Hom}(V^*,\mathcal{M})
\end{equation}

\begin{defn}\label{D31v}
A pairing $\beta:C\otimes A\longrightarrow k$ is said to be $\mathfrak S$-rational if for any $\mathcal M \in \mathfrak S=\mathfrak S_k$, the morphism
\begin{equation}
\alpha_\mathcal{M}:\mathcal M \otimes C \longrightarrow \underline{Hom}(A, \mathcal M),
\end{equation}
corresponding by adjunction to the morphism $\mathcal M\otimes \beta: \mathcal{M} \otimes C \otimes A \longrightarrow \mathcal{M}$  is a monomorphism in $\mathfrak S$.
\end{defn}

Using the definition in \eqref{str3d}, we note that for $\mathcal M\in \mathfrak S^C$, the structure map of ${_A}{\mathfrak I}^C(\mathcal M)\in {_A}\mathfrak S$ can also be given by the composition $\alpha_\mathcal{M} \circ \Delta_\mathcal{M}$.

\begin{lem}\label{Lfull}
Let $\beta:C\otimes A\longrightarrow k$ be an $\mathfrak S$-rational pairing. Then, ${_A}{\mathfrak I}^C:\mathfrak S^C \longrightarrow {_A}\mathfrak S$ is a full embedding of categories.
\end{lem}
\begin{proof}
Let $\phi:\mathcal{M} \longrightarrow \mathcal{N}$ be a morphism in ${_A}\mathfrak S$ with $\mathcal{M}, \mathcal{N} \in \mathfrak S^C$. We consider the diagram
$$\begin{tikzcd}[column sep=2cm]
\mathcal{M} \arrow{r}{\phi} \arrow[swap]{d}{\Delta_\mathcal{M}}  & \mathcal{N} \arrow{d}{\Delta_\mathcal{N}}\\
\mathcal{M} \otimes C \arrow{r}{\phi \otimes C}  \arrow[swap]{d}{\alpha_\mathcal{M}}& \mathcal{N} \otimes C \arrow{d}{\alpha_\mathcal{N}}\\
\underline{Hom}(A,\mathcal{M})  \arrow{r}{\underline{Hom}(A,\phi)}&  \underline{Hom}(A,\mathcal{N}) 
\end{tikzcd}$$
Since $\phi\in {_A}\mathfrak S(\mathcal M,\mathcal N)$, we note that
\begin{equation*}
\underline{Hom}(A,\phi)\circ \alpha_\mathcal{M} \circ \Delta_\mathcal{M}=\alpha_\mathcal{N} \circ \Delta_\mathcal{N} \circ \phi
\end{equation*}
From Definition \ref{D31v}, we observe that
\begin{equation*}
\alpha_\mathcal{N} \circ (\phi \otimes C)=\underline{Hom}(A,\phi) \circ \alpha_\mathcal{M}
\end{equation*}
Since $\alpha_\mathcal N$ is a monomorphism, we now obtain $(\phi \otimes C) \circ \Delta_\mathcal{M}=\Delta_\mathcal{N} \circ \phi$, i.e., $\phi$ is a morphism in $\mathfrak S^C$.
\end{proof}

\begin{prop}\label{P3.3qo}
The category $\mathfrak S^C$ is closed under subobjects, direct sums and quotients in ${_A}\mathfrak S$.
\end{prop}
\begin{proof}
Let $\mathcal M \in \mathfrak S^C$ and $\mathcal N$ be a subobject of $\mathcal M$ in ${_A}\mathfrak S$.   We will show that $\mathcal N \in \mathfrak S^C$. We consider the structure map  $\mu'_\mathcal N:\mathcal N \longrightarrow \underline{Hom}(A,\mathcal N)$. Since $k$ is a field (hence $\_\_\otimes C:\mathfrak S\longrightarrow \mathfrak S$ is exact), we now have the following commutative diagram with exact rows in $\mathfrak S$
$$\begin{tikzcd}[column sep=2cm]
0 \arrow{r} & \mathcal N \arrow{r}{\iota} & \mathcal M \arrow{r}{\pi} \arrow{d}{\Delta_\mathcal M}& \mathcal M/\mathcal N \arrow{r} & 0\\
0 \arrow{r} & \mathcal N \otimes C \arrow{r}{\iota\otimes C} \arrow{d}[swap]{\alpha_\mathcal N}& \mathcal M \otimes C \arrow{r}{\pi \otimes C} \arrow[swap]{d}{\alpha_\mathcal M}& \mathcal M/\mathcal N \otimes C \arrow{r} \arrow{d}{\alpha_{\mathcal M/\mathcal N}}& 0\\
0 \arrow{r} & \underline{Hom}(A,\mathcal N) \arrow{r}{\underline{Hom}(A,\iota)} & \underline{Hom}(A,\mathcal M) \arrow{r}{\underline{Hom}(A,\pi)} & \underline{Hom}(A,\mathcal M/\mathcal N)  
\end{tikzcd}$$
We claim that $(\pi \otimes C) \circ \Delta_\mathcal M \circ \iota=0$. 
Using the fact that $\iota:\mathcal N \longrightarrow \mathcal M$ is a morphism in ${_A}\mathfrak S$, we have
\begin{equation*}
\alpha_{\mathcal M/\mathcal N} \circ (\pi \otimes C)\circ \Delta_\mathcal M \circ \iota=\underline{Hom}(A,\pi) \circ \alpha_\mathcal M \circ \Delta_\mathcal M \circ \iota=\underline{Hom}(A,\pi)  \circ \underline{Hom}(A,\iota) \circ \mu'_\mathcal N=0
\end{equation*}
Since $\alpha_{\mathcal M/\mathcal N}$ is a monomorphism, it follows that 
$(\pi \otimes C) \circ \Delta_\mathcal M \circ \iota=0$. Therefore, $\Delta_\mathcal M \circ \iota$ factors through the kernel $\mathcal N \otimes C \xrightarrow{\iota\otimes C} \mathcal M \otimes C$ and we obtain the map $\mathcal N \longrightarrow \mathcal N \otimes C$ which makes $\mathcal N$ a $C$-comodule object in $\mathfrak S$. This shows that $\mathfrak S^C$ is closed under subobjects. From this, it also follows that $\mathfrak S^C$ is closed under quotients. It may be verified directly that  $\mathfrak S^C$ is also closed under direct sums.
\end{proof}

For $\mathcal N \in {_A}\mathfrak S$, we now define
\begin{equation}\label{srat5}
{_A}\mathfrak R^C(\mathcal N):=\sum\limits_{\mbox{\small $\phi\in {_A}\mathfrak S(\mathcal M,\mathcal N)$, $\mathcal M \in \mathfrak S^C$}} Im(\phi:\mathcal M \longrightarrow \mathcal N)
\end{equation}
Since $\mathfrak S^C$ is closed under subobjects, quotients and direct sums in ${_A}\mathfrak S$, we see that ${_A}\mathfrak R^C(\mathcal N) \in \mathfrak S^C$.  We remark that since ${_A}\mathfrak S$ is a Grothendieck category, the subobjects of $\mathcal N \in {_A}\mathfrak S$ always form a set. Accordingly, the expression in 
\eqref{srat5} may be treated as a sum over a set.

\begin{Thm}\label{Thm3.4c}
Let $\beta:C\otimes A\longrightarrow k$ be an $\mathfrak S$-rational pairing. Then, the association $\mathcal N \mapsto {_A}\mathfrak R^C(\mathcal N)$ defines a functor 
${_A}\mathfrak R^C:{_A}\mathfrak S \longrightarrow \mathfrak S^C$, and it is right adjoint to the embedding ${_A}{\mathfrak I}^C:\mathfrak S^C \longrightarrow {_A}\mathfrak S$.
\end{Thm}
\begin{proof}
Let $\phi :\mathcal N\longrightarrow \mathcal N'$ be a morphism in ${_A}\mathfrak S$. Since ${_A}\mathfrak R^C(\mathcal N) \in \mathfrak S^C$, it is clear from the definition in 
\eqref{srat5} that the composition ${_A}\mathfrak R^C(\mathcal N) \hookrightarrow 
\mathcal N\xrightarrow{\phi}\mathcal N'$ factors through ${_A}\mathfrak R^C(\mathcal N')$.  As such, we have an induced map ${_A}\mathfrak R^C(\phi):{_A}\mathfrak R^C(\mathcal N) \longrightarrow {_A}\mathfrak R^C(\mathcal N')$ in $\mathfrak S^C$. We will now show that 
\begin{equation}\label{adjexp}
{_A}\mathfrak S({_A}\mathfrak I^C(\mathcal M),\mathcal N) \cong \mathfrak S^C(\mathcal M,{_A}\mathfrak R^C(\mathcal N))
\end{equation}
for any $\mathcal M \in \mathfrak S^C$ and $\mathcal N \in {_A}\mathfrak S$. Let $\psi \in {_A}\mathfrak S({_A}\mathfrak I^C(\mathcal M),\mathcal N) $. Then, it is clear from the definition in \eqref{srat5} that $\psi$ factors through ${_A}\mathfrak R^C(\mathcal N)$. On the other hand, let $\psi' \in \mathfrak S^C(\mathcal M,{_A}\mathfrak R^C(\mathcal N))$. Composing with the inclusion ${_A}\mathfrak R^C(\mathcal N) \hookrightarrow \mathcal N$ in ${_A}\mathfrak S$, we obtain a morphism  in $\mathfrak S_A$. It may be easily verified that these two associations are inverse to each other.
\end{proof}

We recall that a Grothendieck category $\mathfrak S$ is locally finitely generated (see, for instance, \cite{AR}) if it has a set of finitely generated generators. Equivalently, every object of $\mathfrak S$ may be expressed a filtered colimit of its finitely generated subobjects. We  have noted before that since $k$ is a field, the  bifunctor $\_\_\otimes \_\_: \mathfrak S\times k-Mod\longrightarrow \mathfrak S$ is exact in both variables. In particular, for any $\mathcal M\in \mathfrak S$, the functor $\mathcal M\otimes \_\_: k-Mod \longrightarrow \mathfrak S$  preserves finite limits.

\begin{thm}\label{P3.5x}
Suppose that $\mathfrak S$ is locally finitely generated. Suppose also that for any finitely generated $\mathcal M\in \mathfrak S$, the functor $\mathcal M\otimes \_\_: k-Mod \longrightarrow \mathfrak S$  preserves limits. Then, the canonical pairing $C\otimes C^\ast\longrightarrow k$ is $\mathfrak S$-rational. 
\end{thm}

\begin{proof}
We consider a linear map $f:C\longrightarrow k$ in $C^*$ and its dual $f^*:k^*=k\longrightarrow C^*$. Accordingly, for any $\mathcal M\in \mathfrak S$, we have a commutative diagram
\begin{equation}\label{cd3.2f}
\begin{CD}
\mathcal M\otimes C @>\alpha_{\mathcal M}>> \underline{Hom}(C^*,\mathcal M) \\
@V{\mathcal M\otimes f}VV  @VV{\underline{Hom}(f^*,\mathcal M)}V\\
\mathcal M @>\cong>> \underline{Hom}(k^*,\mathcal M)=\mathcal M\\
\end{CD}
\end{equation} It follows from \eqref{cd3.2f} that $Ker(\alpha_{\mathcal M})\subseteq Ker(\mathcal M\otimes f)$ for each $f\in C^*$. Since $\mathcal M\otimes \_\_: k-Mod \longrightarrow \mathfrak S$ is exact, we have $Ker(\alpha_{\mathcal M})\subseteq Ker(\mathcal M\otimes f)=\mathcal M\otimes Ker(f)$. 

\smallskip
We now suppose that $\mathcal M\in \mathfrak S$ is finitely generated. By assumption, $\mathcal M\otimes \_\_: k-Mod \longrightarrow \mathfrak S$  preserves limits and we obtain
\begin{equation}\label{3.3hq}
Ker(\alpha_{\mathcal M})\subseteq \underset{f\in C^*}{\bigcap}\textrm{ }( \mathcal M\otimes Ker(f))=\mathcal M\otimes \left(\underset{f\in C^*}{\bigcap}\textrm{ }Ker(f)\right)=0
\end{equation} Hence, $\alpha_{\mathcal M}$ is a monomorphism whenever  $\mathcal M\in \mathfrak S$ is finitely generated. Then, for $\mathcal N\in \mathfrak S$ and $\mathcal M
\subseteq \mathcal N$ finitely generated, we have a monomorphism
\begin{equation}\label{3.4ft}
\mathcal M\otimes C \xrightarrow{\alpha_{\mathcal M}}\underline{Hom}(C^*,\mathcal M) \hookrightarrow \underline{Hom}(C^*,\mathcal N) 
\end{equation} We note here that $\underline{Hom}(C^*,\mathcal M) \hookrightarrow \underline{Hom}(C^*,\mathcal N) $ is a monomorphism because $\underline{Hom}(C^*,\_\_)$ is a right adjoint. Taking the filtered colimit in $\mathfrak S$ of the monomorphisms appearing in \eqref{3.4ft}, we have the monomorphism $\alpha_{\mathcal N}$ for each $\mathcal N\in \mathfrak S$. 
\end{proof}

Let  $\beta:C\otimes A\longrightarrow k$ be  a pairing that is $\mathfrak S$-rational. For $\mathcal N\in {_A}\mathfrak S$, we have defined ${_A}\mathfrak R^C(\mathcal N)$ in \eqref{srat5} to be the sum of images of $\phi\in {_A}\mathfrak S(\mathcal M,\mathcal N)$ with $\mathcal M\in \mathfrak S^C$. We conclude this section by providing another description of the functor ${_A}\mathfrak R^C$.

\begin{thm}\label{P5.1}
Let $\beta:C\otimes A\longrightarrow k$ be an $\mathfrak S$-rational pairing. Let $\mathcal N \in {_A}\mathfrak S$. Then ${_A}\mathfrak R^C(\mathcal N)$ may be expressed as the pullback
\begin{equation}\label{pull5}
\begin{CD}
{_A}\mathfrak R^C(\mathcal N) @>>> \mathcal N\\
@VVV @VV{\mu'_{\mathcal N}}V\\
\mathcal N\otimes C @>\alpha_{\mathcal N}>> \underline{Hom}(A,\mathcal N)\\
\end{CD}
\end{equation}
\end{thm} 

\begin{proof}
We begin by setting $\mathcal T\in \mathfrak S_A$ to be the pullback $\mathcal T:=lim(\mathcal N\otimes C \longrightarrow \underline{Hom}(A,\mathcal N)\longleftarrow \mathcal N)$ with induced morphisms $\iota:\mathcal T\longrightarrow \mathcal N$ and $\iota':\mathcal T\longrightarrow \mathcal N\otimes C$. Let $\mathcal M\in \mathfrak S^C$ and consider $\phi\in {_A}\mathfrak S(\mathcal M,\mathcal N)$. Then, we have commutative diagrams
\begin{equation}\label{cd5.2}
\begin{array}{c}
 \begin{CD}
 \mathcal M @>\Delta_{\mathcal M}>>  \mathcal M\otimes C @>\alpha_{\mathcal M}>>\underline{Hom}(A,\mathcal M) \\
 @. @VV{\phi\otimes C}V @VV\underline{Hom}(A,\phi)V \\
 @. \mathcal N\otimes C@>\alpha_{\mathcal N}>>\underline{Hom}(A,\mathcal N)\\
 \end{CD}\qquad\qquad\qquad\qquad \qquad \begin{CD}
\mathcal M @>\phi>> \mathcal N\\
@V{\mu'_{\mathcal M}}VV @VV{\mu'_{\mathcal N}}V \\
 \underline{Hom}(A,\mathcal M)@> \underline{Hom}(A,\phi)>>  \underline{Hom}(A,\mathcal N)\\
 \end{CD}
\end{array}
\end{equation} Since $\alpha_{\mathcal M}\circ \Delta_{\mathcal M}=\mu'_{\mathcal M}$, it follows from \eqref{cd5.2} that $\mathcal M$ admits an induced morphism to the pullback $\mathcal T$. Since $\beta:C\otimes A\longrightarrow k$ is $\mathfrak S$-rational, $\alpha_{\mathcal N}$ is a monomorphism and hence so is its pullback $\iota:\mathcal T\longrightarrow \mathcal N$. By the definition in \eqref{srat5}, it now follows that ${_A}\mathfrak R^C(\mathcal N)\subseteq \mathcal T\subseteq \mathcal N$. To complete the proof, we will show that $\mathcal T$ is itself a $C$-comodule object in $\mathfrak S$. 
For this, we consider
$$\begin{tikzcd}[column sep=2cm]
 &   & \mathcal{T}   \arrow{d}{\iota'}&  & \\
0 \arrow{r} & \mathcal{T} \otimes C \arrow{r}{\iota \otimes C} \arrow{d}[swap]{\alpha_\mathcal{T}}& \mathcal N \otimes C \arrow{r}{\pi \otimes C} \arrow[swap]{d}{\alpha_\mathcal N}& \mathcal N/\mathcal{T} \otimes C \arrow{r} \arrow{d}{\alpha_{\mathcal N/\mathcal{T}}}& 0\\
0 \arrow{r} & \underline{Hom}(A,\mathcal{T}) \arrow{r}{\underline{Hom}(A,\iota)} & \underline{Hom}(A,\mathcal N) \arrow{r}{\underline{Hom}(A,\pi)} & \underline{Hom}(A,\mathcal N/\mathcal{T})  
\end{tikzcd}$$ We have
\begin{equation}\label{eq5.3}
\alpha_{\mathcal N/\mathcal{T}}\circ (\pi\otimes C)\circ \iota'=\underline{Hom}(A,\pi)\circ \alpha_{\mathcal N}\circ \iota'=\underline{Hom}(A,\pi)\circ\mu'_{\mathcal N}\circ \iota=\underline{Hom}(A,\pi)\circ \underline{Hom}(A,\iota)\circ \mu'_{\mathcal T}=0
\end{equation} Since $\alpha_{\mathcal N/\mathcal{T}}$ is a monomorphism, we obtain $(\pi\otimes C)\circ \iota'=0$. This gives an induced morphism $\Delta_{\mathcal T}:\mathcal T\longrightarrow 
\mathcal T\otimes C$ which makes $\mathcal T$ a $C$-comodule object in $\mathfrak S$. 
\end{proof}

\section{Generators of $C$-comodule objects in $\mathfrak S$}

\smallskip We continue with $C$ being a $k$-coalgebra. In this section, we will study when $\mathfrak S^C$ is a Grothendieck category and when it is locally finitely generated.
For any object $\mathcal N\in \mathfrak S$, we note that the comultiplication on $C$ makes $\mathcal N\otimes C$ into a right $C$-comodule object in $\mathfrak S$ in a canonical manner. 

\begin{lem}\label{L4.1}
Let $(\mathcal M,\Delta_{\mathcal M})\in \mathfrak S^C$. Then, $\mathcal M$ is a subobject of $\mathcal M\otimes C$ in the category $\mathfrak S^C$.
\end{lem}

\begin{proof}
We note that $\Delta_{\mathcal M}:\mathcal M\longrightarrow \mathcal M\otimes C$ is a morphism in $\mathfrak S^C$. If $\epsilon: C\longrightarrow k$ is the counit of $C$, we know that $(\mathcal M\otimes \epsilon)\circ \Delta_{\mathcal M}$ is the identity in $\mathfrak S$. Hence, $\Delta_{\mathcal M}$ is a monomorphism in $\mathfrak S$. Since kernels of morphisms in 
$\mathfrak S^C$ are computed in $\mathfrak S$, the result follows.  

\end{proof}

\begin{lem}\label{L4.2}
The category $\mathfrak S^C$ is well-powered, i.e., the collection of subobjects of any given object forms a set.
\end{lem}

\begin{proof}
We consider $(\mathcal M',\Delta_{\mathcal M'})\subseteq (\mathcal M,\Delta_{\mathcal M})$ in $\mathfrak S^C$. Since kernels in $\mathfrak S^C$ are computed in
$\mathfrak S$, we know that $\mathcal M'\subseteq \mathcal M$ in $\mathfrak S$. Additionally, since $\_\_\otimes C$ is exact, we know that $\mathcal M'\otimes C\subseteq 
\mathcal M\otimes C$ in $\mathfrak S$. This means that the factorization of $\mathcal M'\longrightarrow \mathcal M\longrightarrow \mathcal M\otimes C$ through $\mathcal M'\otimes C$ must be unique. Hence, the subobjects of $ (\mathcal M,\Delta_{\mathcal M})$ in $\mathfrak S^C$ correspond to a subcollection of the subobjects of $\mathcal M$ in $\mathfrak S$. Since $\mathfrak S$ is a Grothendieck category, it is well-powered. Hence, the subobjects of $\mathcal M$ in $\mathfrak S$ form a set. The result is now clear.
\end{proof}

\begin{Thm}\label{P4.2}
Let $C$ be a $k$-coalgebra. Then,  $\mathfrak S^C$ is a Grothendieck category.
\end{Thm}
\begin{proof} Since filtered colimits and finite limits in $\mathfrak S^C$ are computed in $\mathfrak S$, it is clear that $\mathfrak S^C$ satisfies (AB5).
Also since $\mathfrak S$ is a Grothendieck category, we can    choose a generator $\mathcal G\in \mathfrak S$. We consider some $(\mathcal M,\Delta_{\mathcal M})\in \mathfrak S^C$. We choose an epimorphism $\pi: \mathcal G^{(\Lambda)}\longrightarrow \mathcal M$ in $\mathfrak S$ from a direct sum of copies of $\mathcal G$. Accordingly, we have an epimorphism $\pi\otimes C:
 \mathcal G^{(\Lambda)}\otimes C=(\mathcal G\otimes C)^{(\Lambda)}\longrightarrow \mathcal M\otimes C$ in $\mathfrak S^C$. We now consider the following pullback diagram in
 $\mathfrak S^C$
 \begin{equation}\label{4.0cd}
 \begin{CD}
 \mathcal N @>>> \mathcal M\\
 @VVV @VV\Delta_{\mathcal M}V \\
 (\mathcal G\otimes C)^{(\Lambda)} @>\pi\otimes C>> \mathcal M\otimes C\\
 \end{CD}
 \end{equation} By Lemma \ref{L4.1}, we know that $\mathcal M\subseteq \mathcal M\otimes C$ in $\mathfrak S^C$. Accordingly, $\mathcal N\subseteq  (\mathcal G\otimes C)^{(\Lambda)} $ in $\mathfrak S^C$. 
 We now let $Fin(\Lambda)$ denote the collection of finite subsets of $\Lambda$. For any $T\in Fin(\Lambda)$, we set 
 \begin{equation}
 \mathcal E_T:=(\mathcal G\otimes C)^{(\Lambda)}\times_{(\mathcal M\otimes C)} Im((\pi\otimes C)|(\mathcal G\otimes C)^{(T)})
 \end{equation}
in $\mathfrak S^C$. 
 We now consider the following two pullback squares in $\mathfrak S^C$
 \begin{equation}\label{4.1cd}
 \begin{array}{lll}
 \begin{CD}
 \mathcal N\cap \mathcal E_T @>>> \mathcal E_T \\
 @VVV @VVV \\
 \mathcal N @>>> (\mathcal G\otimes C)^{(\Lambda)}\\
 \end{CD}
 &\qquad & 
 \begin{CD}
 \mathcal K_T @>>> Im((\pi\otimes C)|(\mathcal G\otimes C)^{(T)})\\
 @VVV @VVV \\
 \mathcal M @>>> \mathcal M\otimes C\\
 \end{CD} \\
 \end{array}
 \end{equation} The left hand side diagram in \eqref{4.1cd} consists of subobjects of $ (\mathcal G\otimes C)^{(\Lambda)}$ and the right hand side diagram consists of subobjects of 
$\mathcal M\otimes C$.  It now follows that all the squares in the following diagram are pullbacks in $\mathfrak S^C$
 \begin{equation}\label{4.4cd}
 \begin{CD}
\mathcal N\cap (\mathcal G\otimes C)^{(T)}@>>>  \mathcal N\cap \mathcal E_T @>>> \mathcal K_T \\
 @VVV @VVV @VVV \\
(\mathcal G\otimes C)^{(T)}@>>>  \mathcal E_T @>>>  Im((\pi\otimes C)|(\mathcal G\otimes C)^{(T)})\\
 \end{CD}
 \end{equation}
By definition, we know that the composition $(\mathcal G\otimes C)^{(T)}\longrightarrow   \mathcal E_T \longrightarrow    Im((\pi\otimes C)|(\mathcal G\otimes C)^{(T)})$ is an epimorphism. Since $\mathfrak S^C$ is abelian, it follows that $\pi_T:\mathcal N\cap (\mathcal G\otimes C)^{(T)}\longrightarrow \mathcal N\cap \mathcal E_T \longrightarrow   \mathcal K_T $ is also an epimorphism. 

\smallskip
Also since $\pi\otimes C:(\mathcal G\otimes C)^{(\Lambda)}\longrightarrow \mathcal M\otimes C$ is an epimorphism, applying filtered colimits to the right hand diagram in \eqref{4.1cd} over all $T\in Fin(\Lambda)$ gives us $\mathcal M=\underset{T\in Fin(\Lambda)}{\bigcup}\mathcal K_T$. Accordingly, the subobjects of $(\mathcal G\otimes C)^n$ in $\mathfrak S^C$, where $n\geq 1$, give a set of generators for $\mathfrak S^C$. By Lemma \ref{L4.2}, we know that $\mathfrak S^C$ is well-powered, and hence this collection is a set. 
 
\end{proof}

\begin{lem}\label{L4.25}
Let $\beta:C\otimes A\longrightarrow k$ be an $\mathfrak S$-rational pairing and let $\mathcal M\in \mathfrak S^C$ be finitely generated as an object of ${_A}\mathfrak S$. Then, $\mathcal M$ is also finitely generated as an object of $\mathfrak S^C$. 
\end{lem}

\begin{proof}
Let $\{\mathcal N_i\}_{i\in I}$ be a filtered system of objects in $\mathfrak S^C$ connected by monomorphisms. Let $\mathcal N=\underset{i\in I}{\varinjlim}\textrm{ }\mathcal N_i\in \mathfrak S^C$. Since filtered colimits in both $\mathfrak S^C$ and ${_A}\mathfrak S$ are computed in $\mathfrak S$,  we note that $\mathcal N$ is also the colimit of this system in ${_A}\mathfrak S$.  For $i\in I$, we let $\xi_i$ denote the inclusion $\xi_i:\mathcal N_i\longrightarrow \mathcal N$. 

\smallskip We now consider a morphism $\phi:\mathcal M\longrightarrow \mathcal N$ in $\mathfrak S^C$. Since $\mathcal M\in {_A}\mathfrak S$ is finitely generated, we can find some $i_0\in I$ such that
$\phi$ factors as $\phi=\xi_{i_0}\circ \psi$ with $\psi:\mathcal M\longrightarrow \mathcal N_{i_0}$ in ${_A}\mathfrak S$. By Lemma \ref{Lfull}, $\mathfrak S^C$ embeds as a full subcategory of ${_A}\mathfrak S$ and hence the morphism $\psi$ lies in $\mathfrak S^C$. This proves the result.
\end{proof}

\begin{thm}\label{P4.3d}
Suppose that $\mathfrak S$ is locally finitely generated. Let $\beta:C\otimes A\longrightarrow k$ be an $\mathfrak S$-rational pairing. Then, $\mathfrak S^C$ is locally finitely generated.
\end{thm}

\begin{proof}
Let $\{\mathcal G_i\}_{i\in I}$ be a set of finitely generated generators for $\mathfrak S$. By \cite[Corollary B3.17]{AZ}, we know that $\{A\otimes \mathcal G_i\}_{i\in I}$ is a set of generators for ${_A}\mathfrak S$ and by \cite[Proposition B5.1]{AZ}, we know that each $A\otimes \mathcal G_i$ is finitely generated in ${_A}\mathfrak S$. Accordingly, ${_A}\mathfrak S$ is locally finitely generated.

\smallskip
We now consider some $\mathcal N\in \mathfrak S^C$. Since $\beta:C\otimes A\longrightarrow k$ is an $\mathfrak S$-rational pairing, it follows from Lemma \ref{Lfull} that $\mathfrak S^C$ may be treated as a full subcategory of ${_A}\mathfrak S$.  We write $fg_A(\mathcal N)$ for the collection of finitely generated subobjects of $\mathcal N$ in ${_A}\mathfrak S$. Since ${_A}\mathfrak S$ is locally finitely generated, we may write $\mathcal N$ as the filtered colimit in ${_A}\mathfrak S$ of all $\mathcal M\in fg_A(\mathcal N)$. By Proposition \ref{P3.3qo}, $\mathfrak S^C$ is closed under subobjects in ${_A}\mathfrak S$ and hence each $\mathcal M\in fg_A(\mathcal N)$ lies in $\mathfrak S^C$. By Lemma \ref{L4.25},  each $\mathcal M\in fg_A(\mathcal N)$ is also finitely generated as an object of $\mathfrak S^C$.  Since filtered colimits  in both 
$\mathfrak S^C$ and ${_A}\mathfrak S$ are both computed in $\mathfrak S$, we now see that $\mathcal N$ is also the filtered colimit in $\mathfrak S^C$ of finitely generated subobjects. This proves the result. 
\end{proof}

\begin{thm}\label{P4.4z} Suppose that $\mathfrak S$ is locally finitely generated. Suppose also that for any finitely generated $\mathcal M\in \mathfrak S$, the functor $\mathcal M\otimes \_\_: k-Mod \longrightarrow \mathfrak S$  preserves limits.  Then, $\mathfrak S^C$ is locally finitely generated.
\end{thm}

\begin{proof}
By applying Proposition \ref{P3.5x}, it follows that the canonical pairing $C\otimes C^*\longrightarrow k$ is $\mathfrak S$-rational. The result now follows from Proposition \ref{P4.3d}. 
\end{proof}

Since $\mathfrak S^C$ is a Grothendieck category, we know that every $C$-comodule object in $\mathfrak S$ has an injective envelope. We conclude this section  with some important properties of injectives in $\mathfrak S^C$.

\begin{lem}\label{L4.5f}
Let $\mathcal E\in \mathfrak S$ be injective. Then, $\mathcal E\otimes C$ is injective in $\mathfrak S^C$.
\end{lem}
\begin{proof}
Since $\mathcal E\in \mathfrak S$ is injective, it follows from the adjunction in \eqref{adj2.7} that the functor $\mathfrak S^C(\_\_,\mathcal E\otimes C)=\mathfrak S(\_\_,\mathcal E)$ is exact. 
\end{proof}

We will say that a $C$-comodule object in $\mathfrak S$ is free if it is of the form $\mathcal M\otimes C$ for some $\mathcal M\in \mathfrak S$. Accordingly, if $\mathcal E\in \mathfrak S$ is injective, we will say that  $\mathcal E\otimes C$ is a free injective in $\mathfrak S^C$.

\begin{lem}\label{L4.6h}
Every $C$-comodule object in $\mathfrak S$ embeds into a free injective in $\mathfrak S^C$.
\end{lem}

\begin{proof}
By Lemma \ref{L4.1}, we know that any $\mathcal M\in \mathfrak S^C$ is a subobject of $\mathcal M\otimes C$ in the category $\mathfrak S^C$. Since $\mathfrak S$ is a Grothendieck category, we can choose an embedding $\mathcal M\hookrightarrow \mathcal E$ with $\mathcal E$ injective. Then, $\mathcal M\hookrightarrow \mathcal M\otimes C\hookrightarrow \mathcal E\otimes C$ gives us the embedding we need.
\end{proof}

\begin{thm}\label{P4.7u}
A $C$-comodule object in $\mathfrak S$  is injective if and only if it is a direct summand of a free injective in $\mathfrak S^C$.
\end{thm}

\begin{proof}
If $\mathcal M\in \mathfrak S^C$ is a direct summand of a free injective, it must be injective in $\mathfrak S^C$. On the other hand, for any $\mathcal M\in 
\mathfrak S^C$, we apply Lemma \ref{L4.6h} to obtain an embedding $\mathcal M\hookrightarrow \mathcal E\otimes C$ in $\mathfrak S^C$, where
$\mathcal E\in \mathfrak S$ is injective. Additionally, if $\mathcal M\in \mathfrak S^C$ is injective, this monomorphism splits, making $\mathcal M$ a direct summand of a free injective in $\mathfrak S^C$.
\end{proof}

\begin{thm}\label{P4.8sy}
Suppose that any direct sum of injectives in $\mathfrak S$ is injective. Then, any direct sum of injective objects in $\mathfrak S^C$ is also 
injective in $\mathfrak S^C$.
\end{thm}

\begin{proof}
Let $\{\mathcal E_i\}_{i\in I}$ be a family of injectives in $\mathfrak S^C$. Using Proposition \ref{P4.7u}, for each $i\in I$, we choose an injective
$\mathcal E'_i\in \mathfrak S$ such that $\mathcal E_i$ is a direct summand of $\mathcal E'_i\otimes C$.  Accordingly, $\underset{i\in I}{\bigoplus}
\textrm{ }\mathcal E_i$ is a direct summand of $\left(\underset{i\in I}{\bigoplus}
\textrm{ }\mathcal E'_i\right)\otimes C$. By assumption, the direct sum $\underset{i\in I}{\bigoplus}
\textrm{ }\mathcal E'_i$ of injectives in $\mathfrak S$ is injective and hence $\left(\underset{i\in I}{\bigoplus}
\textrm{ }\mathcal E'_i\right)\otimes C$ is a free injective. The result now follows by applying Proposition \ref{P4.7u}.
\end{proof}

\begin{cor}
\label{C4.9gr} Let $\mathfrak S$ be a locally noetherian  Grothendieck category and  $C$ be a $k$-coalgebra. Then:

\smallskip
(a) Any direct sum of injective objects in $\mathfrak S^C$ is also 
injective in $\mathfrak S^C$. 

\smallskip
(b) Suppose also that for any finitely generated $\mathcal M\in \mathfrak S$, the functor $\mathcal M\otimes \_\_: k-Mod \longrightarrow \mathfrak S$  preserves limits.  Then, $\mathfrak S^C$ is locally noetherian. 
\end{cor}

\begin{proof}
We know  that the direct sum of injectives in a locally noetherian Grothendieck category is injective. The result of (a) now follows from Proposition \ref{P4.8sy}. To prove (b), we note that any locally noetherian category is also locally finitely generated. Applying Proposition \ref{P4.4z}, we see that $\mathfrak S^C$ is also locally finitely generated. By part (a), we know that the direct sum of injectives in $\mathfrak S^C$ is also injective, and it now follows from \cite[V.4.3]{Sten} that $\mathfrak S^C$ is locally noetherian.
\end{proof}

\section{Relative Hopf modules in $\mathfrak S$}

From now onwards, $H$ will denote a Hopf algebra over $k$, having multiplication $\mu_H:H\otimes H\longrightarrow H$, comultiplication $\Delta_H:
H\longrightarrow H\otimes H$, along with unit $u_H:k\longrightarrow H$ and counit $\epsilon_H:H\longrightarrow k$. For any vector spaces $V$, $W$, we will denote by $T_{V,W}$ the canonical isomorphism $V \otimes W \xrightarrow{\cong} W \otimes V$. Our aim in the rest of this paper is to develop the cohomology theory for $(A,H)$-Hopf module objects in $\mathfrak S$. 

\smallskip
If $(A,\Delta_A^H:A\longrightarrow A\otimes H)$ is a right $H$-comodule algebra and $(\mathcal M,\Delta_{\mathcal M}^H:\mathcal M
\longrightarrow \mathcal M\otimes H)$ is a right $H$-comodule object in $\mathfrak S$, then it  follows that $A \otimes \mathcal M$ is also an $H$-comodule object in $\mathfrak S$ with structure morphism
\begin{equation}\label{5.1wy}
 \Delta_{A \otimes \mathcal M}^H:A \otimes \mathcal M \xrightarrow{\Delta_A^H \otimes \Delta_\mathcal M^H} A \otimes H \otimes \mathcal M \otimes H \xrightarrow{A\otimes T_{H, \mathcal M}\otimes H} A \otimes \mathcal M \otimes H \otimes H \xrightarrow{A \otimes \mathcal M \otimes \mu_H} A \otimes \mathcal M \otimes H
\end{equation} 
 Here, $T_{H, \mathcal M}$ denotes the isomorphism $H \otimes \mathcal{M} \cong \mathcal{M} \otimes H$  in $\mathfrak S_k$. Similarly, if $(\mathcal M,\mu_{\mathcal M}^A:A\otimes \mathcal M\longrightarrow \mathcal M)\in {_A}\mathfrak S$, then $\mathcal M\otimes H$ becomes a left $A$-module object in $\mathfrak S$ via the following action
\begin{equation}\label{5.2wy}
\mu_{\mathcal M\otimes H}^A:A\otimes \mathcal M\otimes H \xrightarrow{\Delta_A^H \otimes \mathcal M\otimes H}A\otimes H\otimes \mathcal M\otimes H\xrightarrow{A\otimes T_{H, \mathcal M}\otimes H} A\otimes
\mathcal M\otimes H\otimes H \xrightarrow{\mu_\mathcal M^A\otimes \mu_H}\mathcal M\otimes H
\end{equation} We are now ready to introduce $(A,H)$-Hopf module  objects in $\mathfrak S$.

\begin{defn}\label{D5.1cs}
Let $A$  be a right $H$-comodule algebra. A left-right relative $(A,H)$-Hopf module  object in $\mathfrak S$ is a triple  $(\mathcal M,\mu_\mathcal M^A,\Delta_{\mathcal M}^H) $  such that 

\smallskip
(1) $(\mathcal M,\mu_{\mathcal M}^A)\in {_A}\mathfrak S$ and $(\mathcal M,\Delta_{\mathcal M}^H)\in \mathfrak S^H$

\smallskip
(2) $\Delta^H_{\mathcal M}\circ \mu^A_{\mathcal M}=(\mu^A_{\mathcal M}\otimes \mu_H)\circ (A\otimes T_{H, \mathcal M}\otimes H)\circ  (\Delta_A^H \otimes \Delta^H_{\mathcal M})$. 

\smallskip In other words, $\mu_\mathcal M^A:A\otimes\mathcal M\longrightarrow \mathcal M$ is a morphism in $\mathfrak S^H$. Equivalently, $\Delta^H_{\mathcal M}:
\mathcal M\longrightarrow\mathcal M\otimes H$ is a morphism in ${_A}\mathfrak S$.

\smallskip  A morphism $\phi:\mathcal M \longrightarrow \mathcal N$ of left-right relative $(A,H)$-Hopf module objects in $\mathfrak S$ is a morphism $\phi:\mathcal M
\longrightarrow \mathcal N$ of objects in $\mathfrak S$ that is compatible with the left $A$-action and the right $H$-coaction. We will denote the category of left-right relative $(A,H)$-Hopf module objects in $\mathfrak S$ by ${_A}\mathfrak S^H$.
\end{defn}

We note that ${_A}\mathfrak S^H$ is an abelian category. Further, filtered colimits and finite limits in ${_A}\mathfrak S^H$ may be computed in $\mathfrak S$.  Our first aim is to show that ${_A}\mathfrak S^H$ is a Grothendieck category. We begin with the following observation. If 
 $(\mathcal M,\mu_{\mathcal M}^A)\in {_A}\mathfrak S$, then, $\mathcal M\otimes H\in  {_A}\mathfrak S^H$ with $A$-module action as in \eqref{5.2wy} and $H$-comodule action
given by $\mathcal M\otimes\Delta_H:\mathcal M\otimes H\longrightarrow \mathcal M\otimes H\otimes H$. In fact, this determines a functor ${_A}\mathfrak S\longrightarrow {_A}\mathfrak S^H$. 

\begin{lem}\label{L5.1x}
Let $(\mathcal M,\mu_\mathcal M^A,\Delta_{\mathcal M}^H) \in  {_A}\mathfrak S^H$. Then, $\Delta_{\mathcal M}^H:\mathcal M\longrightarrow \mathcal M\otimes H$ is a monomorphism in 
$ {_A}\mathfrak S^H$.
\end{lem}

\begin{proof}
Since $(\mathcal M,\mu_\mathcal M^A,\Delta_{\mathcal M}^H) \in  {_A}\mathfrak S^H$, we know that $\Delta_{\mathcal M}^H$ is compatible with the $A$-action. From the coassociativity of the $H$-coaction, we know that $\Delta_{\mathcal M}^H$ is a morphism in $\mathfrak S^H$. Since $(\mathcal M\otimes \epsilon_H)\circ \Delta_{\mathcal M}^H=id$, we see that $\Delta_{\mathcal M}^H$ is a monomorphism in
$\mathfrak S$. Since kernels in ${_A}\mathfrak S^H$ may be computed in $\mathfrak S$, the result follows. 
\end{proof}

\begin{lem}\label{wellp5x}
The category ${_A}\mathfrak S^H$ is well-powered, i.e., the collection of subobjects of any given object forms a set.
\end{lem}

\begin{proof}
Let $\phi:\mathcal M'\hookrightarrow\mathcal M$ be a monomorphism in ${_A}\mathfrak S^H$. Since kernels in ${_A}\mathfrak S^H$ are computed in $\mathfrak S$, it follows that $\mathcal M'\longrightarrow \mathcal M$
is a monomorphism in $\mathfrak S$. Since $\_\_\otimes H$ is exact, we see that $\phi\otimes H:\mathcal M'\otimes H\longrightarrow \mathcal M\otimes H$
is also a monomorphism in $\mathfrak S$. Accordingly, the factorizations of $A\otimes \mathcal M'\xrightarrow{A\otimes \phi} A\otimes \mathcal M\xrightarrow{\mu_{\mathcal M}^A}\mathcal M$ and $\mathcal M'\xrightarrow{\phi}\mathcal M\xrightarrow{\Delta_\mathcal M^H}\mathcal M\otimes H$ through $\mathcal M'$ and $\mathcal M'\otimes H$ respectively are unique. Hence, the subobjects of $\mathcal M$ in ${_A}\mathfrak S^H$ correspond to a subcollection of subobjects of $\mathcal M$ in $\mathfrak S$. Since the Grothendieck category $\mathfrak S$ is well-powered, so is ${_A}\mathfrak S^H$. 
\end{proof}

\begin{Thm}\label{Thm5.3sq}
Let $A$ be a right $H$-comodule algebra. Then, the category ${_A}\mathfrak S^H$ of relative $(A,H)$-Hopf module objects in $\mathfrak S$ is a Grothendieck category.
\end{Thm}

\begin{proof}
We will prove this in a manner similar to Theorem \ref{P4.2}. Since filtered colimits and finite limits in ${_A}\mathfrak S^H$ are computed in $\mathfrak S$, it is clear that ${_A}\mathfrak S^H$ satisfies the (AB5) condition. We already know from \cite{AZ} that ${_A}\mathfrak S$ is a Grothendieck category. We choose a generator $\mathcal G\in {_A}\mathfrak S$.  If $(\mathcal M,\mu_\mathcal M^A,\Delta_{\mathcal M}^H)\in {_A}\mathfrak S^H $, we can choose an indexing set $\Lambda$ and an epimorphism $\pi:\mathcal G^{(\Lambda)}\longrightarrow 
\mathcal M$ in ${_A}\mathfrak S$. Applying $\_\_\otimes H$, we have $(\pi\otimes H):\mathcal G^{(\Lambda)}\otimes H=(\mathcal G\otimes H)^{(\Lambda)}\longrightarrow \mathcal M\otimes H$ which must be an epimorphism in ${_A}\mathfrak S^H$ since it is an epimorphism in $\mathfrak S$.  By Lemma \ref{L5.1x}, we know that $\mathcal M\subseteq \mathcal M\otimes H$ in
${_A}\mathfrak S^H$. This allows us to form the pullbacks
\begin{equation}
\mathcal N:=(\mathcal G\otimes H)^{(\Lambda)}\times_{(\mathcal M\otimes H)}\mathcal M \qquad \mathcal E_T:=(\mathcal G\otimes H)^{(\Lambda)}\times_{(\mathcal M\otimes H)} Im((\pi\otimes H)|(\mathcal G\otimes H)^{(T)})
 \end{equation}
in ${_A}\mathfrak S^H$, where $T\in Fin(\Lambda)$, the set of finite subsets of $\Lambda$. Accordingly, we can form  pullback squares
 \begin{equation}\label{5.1cdtf}
 \begin{array}{c}
 \begin{array}{lll}
 \begin{CD}
 \mathcal N\cap \mathcal E_T @>>> \mathcal E_T \\
 @VVV @VVV \\
 \mathcal N @>>> (\mathcal G\otimes H)^{(\Lambda)}\\
 \end{CD}
 &\qquad & 
 \begin{CD}
 \mathcal K_T @>>> Im((\pi\otimes H)|(\mathcal G\otimes H)^{(T)})\\
 @VVV @VVV \\
 \mathcal M @>>> \mathcal M\otimes H\\
 \end{CD} \\
 \end{array} \\ \\ \\
 \begin{CD}
\mathcal N\cap (\mathcal G\otimes H)^{(T)}@>>>  \mathcal N\cap \mathcal E_T @>>> \mathcal K_T \\
 @VVV @VVV @VVV \\
(\mathcal G\otimes H)^{(T)}@>>>  \mathcal E_T @>>>  Im((\pi\otimes H)|(\mathcal G\otimes H)^{(T)})\\
 \end{CD} \\
 \end{array}
 \end{equation} Since the composition $(\mathcal G\otimes H)^{(T)}\longrightarrow  \mathcal E_T \longrightarrow Im((\pi\otimes H)|(\mathcal G\otimes H)^{(T)})$ is an epimorphism
 ${_A}\mathfrak S^H$, so is its pullback $\mathcal N\cap (\mathcal G\otimes H)^{(T)}\longrightarrow \mathcal K_T$. Also by \eqref{5.1cdtf}, we see that $\mathcal M$ may be expressed as the filtered colimit in  ${_A}\mathfrak S^H$ of $\mathcal K_T$, as $T$ varies over $Fin(\Lambda)$. As in the proof of Theorem \ref{P4.2}, it now follows that the subobjects of
 $\{(\mathcal G\otimes H)^n\}_{n\geq 1}$ in ${_A}\mathfrak S^H$ give a family of generators for ${_A}\mathfrak S^H$. By Lemma \ref{wellp5x}, ${_A}\mathfrak S^H$ is well-powered, and it now follows that ${_A}\mathfrak S^H$ has  a set of generators. 
\end{proof}

\smallskip
In the rest of this section, we will give an equivalent description of the category ${_A}\mathfrak S^H$ , when $H$ is finite dimensional. For this, we start with  a left $H$-module algebra $B$. Then, one can consider the smash product algebra $B\#H$ which as a vector space is the same as $B \otimes H$, with multiplication given by
\begin{equation*}
(B \otimes H) \otimes (B \otimes H) \xrightarrow{B \otimes \Delta_H \otimes B \otimes H} B \otimes H \otimes H \otimes B \otimes H \xrightarrow{B \otimes H \otimes T_{H,B} \otimes H} B \otimes H \otimes B \otimes H \otimes H \xrightarrow{B \otimes \mu_B^H \otimes \mu_H} B \otimes B \otimes H \xrightarrow{\mu_B \otimes H} B \otimes H
\end{equation*} We note that the associations $b\mapsto (b\# 1)$ and $h\mapsto (1\# h)$ respectively give embeddings of $B$ and $H$ as subalgebras of 
$B\# H$ (see, for instance, \cite[$\S$ 6.1.7]{DNR}). We will now define  $(B,H)$-equivariant Hopf-module  objects in $\mathfrak S$.

\begin{defn}
Let $(B,\mu^H_B:H\otimes B\longrightarrow B)$ be a left $H$-module algebra. A $(B,H)$-equivariant module  object in $\mathfrak S$ is a triple $(\mathcal{M},\mu_\mathcal{M}^B, \mu_\mathcal{M}^H)$ such that

\smallskip
(1) $(\mathcal{M},\mu_\mathcal{M}^B) \in {_B}{\mathfrak S}$ and $(\mathcal{M},\mu_\mathcal{M}^H) \in {_H}{\mathfrak S}$ 

\smallskip
(2) $\mu_{\mathcal M}^B\circ \mu^H_{B\otimes \mathcal M}=\mu_{\mathcal M}^B\circ (\mu_B^H\otimes \mu_{\mathcal M}^H)\circ (H\otimes T_{H,B}\otimes \mathcal M)\circ(\Delta_H\otimes B\otimes \mathcal M)=\mu_\mathcal{M}^H \circ (H \otimes \mu_\mathcal{M}^B)$, where $ \mu^H_{B\otimes \mathcal M}:H\otimes B\otimes \mathcal M\longrightarrow B\otimes  \mathcal M$ is defined by the diagonal action of $H$.

\smallskip
A morphism $\phi:(\mathcal{M},\mu_\mathcal{M}^B, \mu_\mathcal{M}^H)\longrightarrow (\mathcal{N},\mu_\mathcal{N}^B, \mu_\mathcal{N}^H)$ of $(B,H)$-equivariant module  objects in $\mathfrak S$ is a morphism  $\phi:\mathcal M\longrightarrow \mathcal N$ in $\mathfrak S$ that is compatible with both the $B$-action and the $H$-action.  We will denote the corresponding category by ${_{(B,H)}}{Eq(\mathfrak S})$.
\end{defn}

\begin{thm}\label{eqj}
Let $B$ be a left $H$-module algebra. Then, the category $_{(B\#H)}{\mathfrak S}$ is equivalent to the category ${_{(B,H)}}{Eq(\mathfrak S})$. In particular, ${_{(B,H)}}{Eq(\mathfrak S})$ is a Grothendieck category.
\end{thm}
\begin{proof}
We write $i_B:B\hookrightarrow B\#H$ and $i_H:H\hookrightarrow B\#H$ for the respective algebra inclusions.  Hence,  any   $(\mathcal{M},\mu_\mathcal{M}^{B\#H}) \in { _{(B\#H)}}{\mathfrak S}$ is equipped with the structure of both a $B$-module object and an $H$-module object in $\mathfrak S$ by setting $\mu_{\mathcal M}^B:=\mu_\mathcal{M}^{B\#H}\circ (i_B\otimes \mathcal M)$ and $\mu_{\mathcal M}^B:=\mu_\mathcal{M}^{B\#H}\circ (i_H\otimes \mathcal M)$ respectively.  For $(\mathcal{M},\mu_\mathcal{M}^{B\#H}) \in { _{(B\#H)}}{\mathfrak S}$, we now have

\begin{align*}
\mu_\mathcal{M}^H \circ (H \otimes \mu_\mathcal{M}^B)&=\mu_\mathcal{M}^{B\#H}\circ ((B\#H)\otimes \mu_\mathcal{M}^{B\#H}) \circ (i_H\otimes i_B\otimes \mathcal M)\\
&=\mu_\mathcal{M}^{B\#H}\circ (\mu_{B\#H}\otimes\mathcal M) \circ (i_H\otimes i_B\otimes \mathcal M) \\
&=\mu_\mathcal{M}^{B\#H}\circ (\mu^H_B\otimes H\otimes \mathcal M)\circ (H\otimes T_{H,B}\otimes \mathcal M)\circ (\Delta_H\otimes B\otimes \mathcal M)\\
&=\mu_\mathcal{M}^{B\#H}\circ (\mu_{B\#H}\otimes\mathcal M) \circ (i_B\otimes i_H\otimes \mathcal M)\circ (\mu^H_B\otimes H\otimes \mathcal M)\circ (H\otimes T_{H,B}\otimes \mathcal M)\circ (\Delta_H\otimes B\otimes \mathcal M)\\
&= \mu_\mathcal{M}^{B\#H}\circ ((B\#H)\otimes \mu_\mathcal{M}^{B\#H})\circ (i_B\otimes i_H\otimes \mathcal M)\circ (\mu^H_B\otimes H\otimes \mathcal M)\circ (H\otimes T_{H,B}\otimes \mathcal M)\circ (\Delta_H\otimes B\otimes \mathcal M)\\
&= \mu_\mathcal{M}^{B\#H}\circ (i_B\otimes \mathcal{M})\circ (B\otimes\mu_{\mathcal M}^H)\circ (\mu^H_B\otimes H\otimes \mathcal M)\circ (H\otimes T_{H,B}\otimes \mathcal M)\circ (\Delta_H\otimes B\otimes \mathcal M)\\
&= \mu_\mathcal{M}^{B}\circ (B\otimes\mu_{\mathcal M}^H)\circ (\mu^H_B\otimes H\otimes \mathcal M)\circ (H\otimes T_{H,B}\otimes \mathcal M)\circ (\Delta_H\otimes B\otimes \mathcal M)\\
&= \mu_{\mathcal M}^B\circ (\mu_B^H\otimes \mu_{\mathcal M}^H)\circ (H\otimes T_{H,B}\otimes \mathcal M)\circ(\Delta_H\otimes B\otimes \mathcal M)
\end{align*}
This shows that $(\mathcal{M},\mu_\mathcal{M}^{B\#H}) \in { _{(B\#H)}}{\mathfrak S}$ may be treated as an object of $ {_{(B,H)}}{Eq(\mathfrak S})$. Similarly, it may be easily verified that any $(\mathcal{N},\mu_\mathcal{N}^B, \mu_\mathcal{N}^H) \in {_{(B,H)}}{Eq(\mathfrak S})$ carries a $(B \#H)$-action given by
\begin{equation*}
(B \#H) \otimes \mathcal N \xrightarrow{B \otimes \mu_\mathcal N^H} B \otimes \mathcal{N} \xrightarrow{\mu_\mathcal{N}^B} \mathcal{N} 
\end{equation*}
\end{proof}

We now let $H$ be a finite dimensional Hopf algebra. Then, the linear dual $H^*$ is also a Hopf algebra. In particular, if $A$ is a right $H$-comodule algebra, we note that $A$ becomes a left $H^*$-module algebra and vice-versa.

\begin{Thm}
Let $H$ be a finite dimensional Hopf algebra. Let $A$ be a right $H$-comodule algebra. Then, the category ${_A}{\mathfrak S}^H$ is equivalent to the category ${ _{(A\#H^*)}}{\mathfrak S}$.
\end{Thm}
\begin{proof}
Let $(\mathcal M,\mu_{\mathcal M}^A,\mu_{\mathcal M}^H)$ be an object in ${_A}{\mathfrak S}^H$. Then, we have $(\mathcal M,\mu_\mathcal M^A) \in {_A}{\mathfrak S}$ and $(\mathcal M,\Delta_\mathcal M^H) \in \mathfrak S^H$ such that
\begin{equation}\label{xx}
\Delta^H_{\mathcal M}\circ \mu^A_{\mathcal M}=(\mu^A_{\mathcal M}\otimes \mu_H)\circ (A\otimes T_{H, \mathcal M}\otimes H)\circ  (\Delta_A^H \otimes \Delta^H_{\mathcal M})=(\mu_\mathcal M^A \otimes H) \circ \Delta_{A \otimes \mathcal M}^H
\end{equation} where $ \Delta_{A \otimes \mathcal M}^H$ is as defined in \eqref{5.1wy}.
We now make $\mathcal M$ a left $H^*$-module object in $\mathfrak S$ by setting 
\begin{equation*}
\mu_{\mathcal M}^{H^*}:H^* \otimes \mathcal M \xrightarrow{H^* \otimes \Delta_\mathcal M^H} H^* \otimes \mathcal M \otimes H \xrightarrow{H^* \otimes T_{\mathcal M,H}} H^* \otimes H \otimes \mathcal M \xrightarrow{ev_H\otimes \mathcal M} \mathcal M
\end{equation*} where $ev_H$ denotes the evaluation map $ev_H:H^*\otimes H\longrightarrow k$.
 Also, using \eqref{xx}, we see that
\begin{align*}
\mu_\mathcal M^{H^*} \circ (H^* \otimes \mu_\mathcal M^A)&=(ev_H \otimes \mathcal M)(H^* \otimes T_{\mathcal M,H})(H^* \otimes \Delta_\mathcal M^H)(H^* \otimes \mu_\mathcal M^A)\\
&=(ev_H \otimes \mathcal M)(H^* \otimes T_{\mathcal M,H})(H^* \otimes (\Delta_\mathcal M^H \circ \mu_\mathcal M^A))\\
&=(ev_H \otimes \mathcal M)(H^* \otimes T_{\mathcal M,H})(H^* \otimes \mu_\mathcal M^A \otimes H)(H^* \otimes \Delta^H_{A \otimes \mathcal M})\\
&=\mu_\mathcal M^A (ev_H \otimes A \otimes \mathcal M)(H^* \otimes T_{A \otimes \mathcal M,H})(H^* \otimes \Delta^H_{A \otimes \mathcal M})\\
&=\mu_\mathcal M^A \circ \mu^{H^*}_{A \otimes \mathcal M}
\end{align*}
Combining with Proposition \ref{eqj}, this proves that $\mathcal M$ may be treated as an object  of $ { _{A\#H^*}}{\mathfrak S}$. Conversely, we consider $(\mathcal N, \mu^{A\#H^*}_{\mathcal N})\in { _{A\#H^*}}{\mathfrak S}$. Using Proposition \ref{eqj}, we see that  $(\mathcal N, \mu^{A\#H^*}_{\mathcal N})$ may be treated as an $(A,H^*)$-equivariant module object $(\mathcal N,\mu^A_{\mathcal N},\mu^{H^*}_{\mathcal N})$ in $\mathfrak S$. The structure morphism $\mu^{H^*}_\mathcal N:H^* \otimes \mathcal N \longrightarrow \mathcal N$ corresponds by adjunction to a morphism $\Delta^{H}_{\mathcal N}:\mathcal N\longrightarrow  \underline{Hom}(H^*,\mathcal N)\cong \mathcal N\otimes H$, where the last isomorphism follows from \eqref{fd} because  $H$ is finite dimensional. It may now be verified that $(\mathcal N,\mu_{\mathcal N}^A,\Delta_{\mathcal N}^H)\in {_A}\mathfrak S^H$. 
\end{proof}

\section{Cohomology of relative Hopf modules in $\mathfrak S$}

We continue with $A$ being a right $H$-comodule algebra. We denote by $Com-H$ the category of right $H$-comodules.  Since $H$ is a Hopf algebra over a field $k$, we know that $Com-H$ is a Grothendieck category.

\begin{thm}\label{P6.1ku}
Let $(\mathcal M,\mu_{\mathcal M}^A,\Delta_{\mathcal M}^H)\in {_A}\mathfrak S^H$. Then, 
$\mathcal M$ determines a functor $\mathcal M\otimes\_\_:Com-H\longrightarrow {_A}\mathfrak S^H$ that has a right adjoint ${_{{_A}\mathfrak S}}HOM(\mathcal M,\_\_):{_A}
 \mathfrak S^H\longrightarrow Com-H$. In other words, we have natural isomorphisms
 \begin{equation}\label{6.1eqx}
 {_A}\mathfrak S^H(\mathcal M\otimes N,\mathcal P)\cong Com-H(N,{_{{_A}\mathfrak S}}HOM(\mathcal M,\mathcal P))
 \end{equation} for $\mathcal P\in {_A}\mathfrak S^H$ and $N\in Com-H$.
\end{thm}
\begin{proof} 
The object $\mathcal M\otimes N$ carries a left $A$-action given by $\mu_{\mathcal M}^A\otimes N: (A\otimes \mathcal M)\otimes N\longrightarrow  \mathcal M\otimes N$. If $\Delta_N:N\longrightarrow N\otimes H$ is the right $H$-comodule structure map of $N$, then 
\begin{equation}
\Delta^H:{\mathcal M\otimes N}:\mathcal M\otimes N \xrightarrow{\Delta_{\mathcal M}^H\otimes \Delta_N} \mathcal M\otimes H\otimes N\otimes H \xrightarrow{\mathcal M\otimes T_{H,N}\otimes H}\mathcal M\otimes N\otimes H
\otimes H\xrightarrow{\mathcal M\otimes N\otimes \mu_H} \mathcal M\otimes N\otimes H
\end{equation} makes $\mathcal M\otimes N$ a right $H$-comodule object in $\mathfrak S$. It may be directly verified that this makes $\mathcal M\otimes N$ an object of ${_A}\mathfrak S^H$. It is also clear that   $\mathcal M\otimes\_\_$ preserves colimits. By Theorem \ref{Thm5.3sq}, we know that ${_A}\mathfrak S^H$ is a Grothendieck category. We have also noted that
the category $Com-H$ of right $H$-comodules  is a Grothendieck category. Applying therefore \cite[Proposition 8.3.27(iii)]{KS}, it follows that  $\mathcal M\otimes\_\_$  has a right adjoint which we denote by ${_{{_A}\mathfrak S}}HOM(\mathcal M,\_\_):{_A}
 \mathfrak S^H\longrightarrow Com-H$. 
\end{proof}

If $(N,\Delta_N:N\longrightarrow N\otimes H)$ is a right $H$-comodule, we recall that the space of $H$-coinvariants of $N$, denoted $N^{co H}$, is given by $N^{co H}:=\{\mbox{$n\in N$ $\vert$ $\Delta_N(n)=n\otimes 1_H$}\}$. We note that the space $N^{co H}$ may be alternatively written as $N^{co H}=Com-H(k,N)$, where $k$ is treated as an $H$-comodule by means of the structure map
$k\cong k\otimes k\xrightarrow{k\otimes u_H} k\otimes H$. 

\begin{lem}\label{L6.2eb} (a) Let $\mathcal M$, $\mathcal N\in {_A}\mathfrak S^H$. Then, ${_A}\mathfrak S^H(\mathcal M,\mathcal N)={_{{_A}\mathfrak S}}HOM(\mathcal M,\mathcal N)^{co H}$. 

\smallskip
(b) The functor  ${_{{_A}\mathfrak S}}HOM(\mathcal M,\_\_):{_A}
 \mathfrak S^H\longrightarrow Com-H$ preserves injective objects.
\end{lem}
\begin{proof}
(a)  This follows immediately from \eqref{6.1eqx} by setting $N=k$ and using the expression for coinvariants of a right $H$-comodule. 

\smallskip
(b) We have noted in Section 2 that $\_\_\otimes \_\_:\mathfrak S\times k-Mod\longrightarrow \mathfrak S$ is exact in both variables. Accordingly, the left adjoint $\mathcal M\otimes\_\_:Com-H\longrightarrow {_A}\mathfrak S^H$  is exact. Hence, the right adjoint ${_{{_A}\mathfrak S}}HOM(\mathcal M,\_\_):{_A}
 \mathfrak S^H\longrightarrow Com-H$ preserves injectives.
\end{proof}

\begin{Thm}\label{T6.3spec}
Let $(\mathcal M,\mu_{\mathcal M}^A,\Delta_{\mathcal M}^H)$ be a relative $(A,H)$-Hopf module object in $\mathfrak S$. We consider the following functors
\begin{equation}
\begin{array}{c}
\mathfrak  F:={_{{_A}\mathfrak S}}HOM(\mathcal M,\_\_):{_A}\mathfrak S^H\longrightarrow Com-H \qquad \mathcal N\mapsto{_{{_A}\mathfrak S}}HOM(\mathcal M,\mathcal N)\\
\mathfrak G:=(-)^{co H}:Com-H\longrightarrow k-Mod \qquad N\mapsto N^{co H}
\end{array}
\end{equation} Then, we have a spectral sequence
\begin{equation}
E_2^{pq}=R^p(-)^{co H}(R^q{_{{_A}\mathfrak S}}HOM(\mathcal M,\_\_)(\mathcal N)) \Rightarrow (R^{p+q}{_A}\mathfrak S^H(\mathcal M,\_\_))(\mathcal N)
\end{equation}
\end{Thm}

\begin{proof}
The categories ${_A}\mathfrak S^H$, $Com-H$ and $k-Mod$ are all Grothendieck categories, and hence they have enough injectives. By Lemma \ref{L6.2eb}(a), we have 
$(\mathfrak G\circ \mathfrak F)(\mathcal N)={_A}\mathfrak S^H(\mathcal M,\mathcal N)$. By Lemma \ref{L6.2eb}(b), the functor $\mathfrak F$ preserves injective objects. The result is now clear from the Grothendieck spectral sequence for composite functors.
\end{proof}

For the rest of this section, we will always assume that $H$ has a bijective antipode $S:H\longrightarrow H$. 
Let $(\mathcal M,\Delta^H_\mathcal M)$ and $(\mathcal N,\Delta^H_\mathcal N)$ be  objects in $\mathfrak S^H$. For $\phi \in \mathfrak S(\mathcal M,\mathcal N)$, we have a morphism $\rho(\phi) \in \mathfrak S(\mathcal M,\mathcal N \otimes H)$ given by
\begin{equation}\label{rho}
\mathcal M \xrightarrow{\Delta^H_\mathcal M} \mathcal M \otimes H \xrightarrow{\mathcal M \otimes S^{-1}} \mathcal M\otimes H  \xrightarrow{\phi \otimes H} \mathcal N \otimes H \xrightarrow{\Delta^H_\mathcal N \otimes H} \mathcal N \otimes H \otimes H \xrightarrow{\mathcal N \otimes \mu_{H^{op}}} \mathcal N \otimes H
\end{equation} where $\mu_{H^{op}}:H\otimes H\longrightarrow H$ denotes the opposite of the multiplication on $H$. We now have the following result.

\begin{thm}\label{6.4fd}
Let $A$ be a right $H$-comodule algebra. Let $(\mathcal M, \mu^A_\mathcal M, \Delta^H_\mathcal M)$ and $(\mathcal N, \mu^A_\mathcal N, \Delta^H_\mathcal N)$ be objects in ${_A}\mathfrak S^H$ and let $\phi \in  {_A}\mathfrak S(\mathcal M,\mathcal N)$. Then, the morphism $\rho(\phi):\mathcal{M} \longrightarrow \mathcal{N} \otimes H$, as defined in \eqref{rho}, lies in
${_A}\mathfrak S(\mathcal M,\mathcal N\otimes H)$.
\end{thm}
\begin{proof} We have
\begin{align*}
\rho(\phi)\circ \mu_\mathcal M^A&=(\mathcal N \otimes \mu_{H^{op}})(\Delta^H_\mathcal N \otimes H)(\phi \otimes H)(\mathcal M \otimes S^{-1}) \Delta^H_\mathcal M \mu_\mathcal M^A\\
& =(\mathcal N \otimes \mu_{H^{op}})(\Delta^H_\mathcal N \otimes H)(\phi \otimes H)(\mathcal M \otimes S^{-1})(\mu^A_{\mathcal M}\otimes \mu_H)(A\otimes T_{H, \mathcal M}\otimes H)(\Delta_A^H \otimes \Delta^H_{\mathcal M})\\
&=(\mathcal N \otimes \mu_{H^{op}})(\Delta^H_\mathcal N \otimes H)(\phi \otimes H)(\mathcal M \otimes S^{-1})(\mathcal M \otimes \mu_H)(\mu_\mathcal M^A \otimes H \otimes H)(A \otimes T_{H,\mathcal M} \otimes H)(\Delta^H_A \otimes \Delta_\mathcal M^H)  \\
&=(\mathcal N \otimes \mu_{H^{op}})(\Delta^H_\mathcal N \otimes H)(\mathcal N \otimes S^{-1})(\mathcal N \otimes \mu_H)(\phi \otimes H \otimes H)(\mu_\mathcal M^A \otimes H \otimes H) (A \otimes T_{H,\mathcal M} \otimes H)(\Delta^H_A \otimes \Delta_\mathcal M^H) \\
&=(\mathcal N \otimes \mu_{H^{op}})(\Delta^H_\mathcal N \otimes H)(\mathcal N \otimes S^{-1})(\mathcal N \otimes \mu_H)(\mu_\mathcal N^A \otimes H \otimes H)(A \otimes \phi \otimes H \otimes H) (A \otimes T_{H,\mathcal M} \otimes H)(\Delta^H_A \otimes \Delta_\mathcal M^H) \\
&=(\mathcal N \otimes \mu_{H^{op}})(\mathcal N \otimes H \otimes S^{-1})(\Delta_\mathcal N^H \otimes H)(\mu_\mathcal N^A \otimes H)(A \otimes \mathcal N  \otimes \mu_H)(A \otimes \phi \otimes H \otimes H) (A \otimes T_{H,\mathcal M} \otimes H)(\Delta^H_A \otimes \Delta_\mathcal M^H) \\
&=(\mathcal N \otimes \mu_{H^{op}})(\mathcal N \otimes H \otimes S^{-1})(\mu^A_\mathcal N \otimes \mu_H \otimes H)(A \otimes T_{H,\mathcal N} \otimes H \otimes H)(\Delta^H_A \otimes \Delta^H_\mathcal N \otimes H)
  (A \otimes \mathcal N \otimes \mu_H)(A \otimes \phi \otimes H \otimes H)\\
  & \qquad(A \otimes T_{H,\mathcal M} \otimes H)(\Delta^H_A \otimes \Delta_\mathcal M^H) \\
&=(\mathcal N \otimes \mu_{H^{op}})(\mathcal N \otimes H \otimes S^{-1})(\mu^A_\mathcal N \otimes \mu_H \otimes H)(A \otimes T_{H,\mathcal N} \otimes H \otimes H)(A \otimes H \otimes \mathcal N  \otimes H \otimes \mu_H)(A \otimes H \otimes \Delta_\mathcal N^H  \otimes H \otimes H)\\
&\qquad (\Delta_A^H \otimes  \mathcal N \otimes H \otimes H)(A \otimes \phi \otimes H \otimes H)(A \otimes T_{H,\mathcal M} \otimes H)(\Delta^H_A \otimes \mathcal M \otimes H)(A \otimes \Delta^H_\mathcal M) \\
&=(\mathcal N \otimes \mu_{H^{op}})(\mu_\mathcal{N}^A \otimes \mu_H \otimes H)(A \otimes T_{H,\mathcal{N}} \otimes H \otimes H)(A \otimes H \otimes \mathcal{N} \otimes H \otimes S^{-1})(A \otimes H \otimes \mathcal N  \otimes H \otimes \mu_H)(A \otimes H \otimes \Delta_\mathcal N^H  \otimes H \otimes H)\\
&\qquad (\Delta_A^H \otimes  \mathcal N \otimes H \otimes H)(A \otimes \phi \otimes H \otimes H)(A \otimes T_{H,\mathcal M} \otimes H)(\Delta^H_A \otimes \mathcal M \otimes H)(A \otimes \Delta^H_\mathcal M) \\
&=(\mathcal N \otimes \mu_{H^{op}})(\mu_\mathcal{N}^A \otimes \mu_H \otimes H)(A \otimes T_{H,\mathcal{N}} \otimes H \otimes H)(A \otimes H \otimes \mathcal{N} \otimes H \otimes (S^{-1}\circ \mu_H))(A \otimes H \otimes \Delta_\mathcal N^H  \otimes H \otimes H)\\
&\qquad (\Delta_A^H \otimes  \mathcal N \otimes H \otimes H)(A \otimes \phi \otimes H \otimes H)(A \otimes T_{H,\mathcal M} \otimes H)(\Delta^H_A \otimes \mathcal M \otimes H)(A \otimes \Delta^H_\mathcal M) 
\end{align*} Using the fact that $S^{-1} \circ \mu_H=\mu_{H^{op}}\circ (S^{-1} \otimes S^{-1})=\mu_{H^{op}} \circ (S^{-1} \otimes H) \circ (H \otimes S^{-1}) $, we now have
\begin{align*}
\rho(\phi)\circ \mu_\mathcal M^A&=(\mathcal N \otimes \mu_{H^{op}})(\mu_\mathcal{N}^A \otimes \mu_H \otimes H)(A \otimes T_{H,\mathcal{N}} \otimes H \otimes H)(A \otimes H \otimes \mathcal{N} \otimes H \otimes \mu_{H^{op}})(A \otimes H \otimes \mathcal{N} \otimes H \otimes  S^{-1} \otimes H)\\
& \qquad (A \otimes H \otimes \mathcal{N} \otimes H \otimes H \otimes S^{-1})(A \otimes H \otimes \Delta_\mathcal N^H  \otimes H \otimes H)(A \otimes H \otimes T_{H,\mathcal{N}} \otimes H)(\Delta_A^H \otimes H \otimes \mathcal{N} \otimes H)(\Delta_A^H \otimes \mathcal{N} \otimes H)\\
&\qquad (A \otimes \phi \otimes H)(A \otimes \Delta^H_\mathcal M) \\
&=(\mathcal N \otimes \mu_{H^{op}})(\mu_\mathcal{N}^A \otimes \mu_H \otimes H)(A \otimes T_{H,\mathcal{N}} \otimes H \otimes H)(A \otimes H \otimes \mathcal{N} \otimes H \otimes \mu_{H^{op}})(A \otimes H \otimes \mathcal{N} \otimes H \otimes  S^{-1} \otimes H)\\
& \qquad (A \otimes H \otimes \mathcal{N} \otimes H \otimes H \otimes S^{-1})(A \otimes H \otimes \Delta_\mathcal N^H  \otimes H \otimes H)(A \otimes H \otimes T_{H,\mathcal{N}} \otimes H)(A \otimes \Delta_H \otimes \mathcal{N} \otimes H)(\Delta_A^H \otimes \mathcal{N} \otimes H)\\
&\qquad (A \otimes \phi \otimes H)(A \otimes \Delta^H_\mathcal M) \\
& =(\mathcal N \otimes \mu_{H^{op}})(\mu_\mathcal{N}^A \otimes \mu_H \otimes H)(A \otimes T_{H,\mathcal{N}} \otimes H \otimes H)(A \otimes H \otimes \mathcal{N} \otimes H \otimes \mu_{H^{op}})(A \otimes H \otimes \mathcal{N} \otimes H \otimes  S^{-1} \otimes H)\\
&\qquad (A \otimes H \otimes \mathcal{N} \otimes T_{H,H} \otimes H)(A \otimes H \otimes T_{H, \mathcal{N}} \otimes H \otimes H)(A \otimes \Delta_H \otimes \mathcal{N} \otimes H \otimes H)(A \otimes H \otimes \Delta_\mathcal{N}^H \otimes H)(A \otimes H \otimes \mathcal{N} \otimes S^{-1})\\
&\qquad (\Delta_A^H \otimes \mathcal{N} \otimes H)(A \otimes \phi \otimes H)(A \otimes \Delta^H_\mathcal M) \\
&=(\mu_\mathcal{N}^A\otimes H)(A \otimes \mathcal{N} \otimes \mu_H)(A \otimes \mathcal N \otimes T_{H,H})(A \otimes \Delta_\mathcal{N}^H \otimes H)\\
&\qquad (A \otimes \mathcal{N} \otimes S^{-1})(A \otimes \varepsilon_H \otimes \mathcal{N} \otimes H)(\Delta_A^H \otimes \mathcal{N} \otimes H)(A \otimes \phi \otimes H)(A \otimes \Delta^H_\mathcal M) \\
& = (\mu_\mathcal N^A\otimes H)(A \otimes \mathcal N \otimes \mu_H)(A \otimes \mathcal N \otimes T_{H,H})(A \otimes \Delta_\mathcal N^H \otimes H)(A  \otimes \phi \otimes H)(A \otimes \mathcal M \otimes S^{-1})(A \otimes \Delta^H_\mathcal M) \\ &= (\mu_\mathcal N^A\otimes H)(A\otimes\rho(\phi))
\end{align*}
This proves the result.
\end{proof}

For $(\mathcal M,\Delta^H_\mathcal M)$, $(\mathcal N,\Delta^H_\mathcal N)$ in $\mathfrak S^H$,  $\phi \in \mathfrak{S}(\mathcal M,\mathcal N)$ and $h^* \in H^*$, we  set
\begin{equation}\label{6.6yc} h^* \cdot \phi:=( \mathcal N \otimes h^*) \circ \rho(\phi) \in \mathfrak S(\mathcal M,\mathcal N)
\end{equation}   making $\mathfrak S(\mathcal M, \mathcal N)$ into a left $H^*$-module.  

\begin{cor}
Let $A$ be an $H$-comodule algebra. Let $(\mathcal M, \mu^A_\mathcal M, \Delta^H_\mathcal M)$ and $(\mathcal N, \mu^A_\mathcal N, \Delta^H_\mathcal N)$ be objects in ${_A}\mathfrak S^H$. Then, the $k$-space ${_A}\mathfrak S(\mathcal M,\mathcal N) \subseteq \mathfrak S(\mathcal M, \mathcal N)$ is an $H^*$-module.
\end{cor}
\begin{proof}
Let $h^* \in H^*$ and $\phi \in {_A}{\mathfrak S}(\mathcal M, \mathcal N)$. We know by Proposition \ref{6.4fd} that $\rho(\phi) \in \mathfrak S(\mathcal M,\mathcal N \otimes H)$ is $A$-linear  and therefore the composition $h^*\cdot \phi=( \mathcal N \otimes h^*) \circ \rho(\phi)$ is also $A$-linear. The result follows.
\end{proof}

We know that the inclusion of the full subcategory $Com-H\hookrightarrow H^*-Mod$ has a right adjoint which we denote by ${_{H^*}}R^H$. Then, for  any left $H^*$-module $M$,  the rational part ${_{H^*}}R^H(M)$ is the largest right $H$-comodule contained in $M$. For   $(\mathcal M, \mu^A_\mathcal M, \Delta^H_\mathcal M)$ and $(\mathcal N, \mu^A_\mathcal N, \Delta^H_\mathcal N)$  in ${_A}\mathfrak S^H$, we will now show that ${_{{_A}{\mathfrak S}}}H
OM(\mathcal M,\mathcal N)$ is the rational part of the $H^*$-module ${_A}\mathfrak S(\mathcal M,\mathcal N)$.

\begin{thm}\label{6.6i}
Let $A$ be a right $H$-comodule algebra, and let $(\mathcal M, \mu^A_\mathcal M, \Delta^H_\mathcal M)$ and $(\mathcal N, \mu^A_\mathcal N, \Delta^H_\mathcal N)$ be objects in ${_A}\mathfrak S^H$. Then,
\begin{equation}
{_{H^*}}R^H\left({_A}\mathfrak S(\mathcal M,\mathcal N)\right) \cong {_{{_A}{\mathfrak S}}}H
OM(\mathcal M,\mathcal N) \hspace{.5cm} \in\textrm{ } Com-H
\end{equation}
\end{thm}
\begin{proof}
By adjunction, we have the following  isomorphisms 
\begin{equation}
Com-H\left(N, {_{{_A}{\mathfrak S}}}HOM(\mathcal M,\mathcal N) \right) \cong {_A}\mathfrak S^H(\mathcal M \otimes N, \mathcal N)\qquad 
Com-H\left(N, {_{H^*}}R^H\left({_A}\mathfrak S(\mathcal M,\mathcal N)\right) \right) \cong  H^*-Mod(N, {_A}\mathfrak S(\mathcal M,\mathcal N))
\end{equation}
for any $N \in Com-H$. Accordingly, it suffices to show that
\begin{equation}
{_A}\mathfrak S^H(\mathcal M \otimes N, \mathcal N) \cong H^*-Mod(N, {_A}\mathfrak S(\mathcal M,\mathcal N))
\end{equation} 
for any $N\in Com-H$. We start with $\varphi: \mathcal M \otimes N \longrightarrow \mathcal N$ in ${_A}\mathfrak S^H$. We define $\tilde{\varphi}:N \longrightarrow {_A}\mathfrak S(\mathcal M,\mathcal N)$ given by $n \mapsto \tilde{\varphi}(n)$,  where $\tilde{\varphi}(n)$ is defined as 
\begin{equation*}
\tilde{\varphi}(n):\mathcal M \cong \mathcal M \otimes k \xrightarrow{\mathcal M \otimes t_n} \mathcal M \otimes N \xrightarrow{\varphi} \mathcal N
\end{equation*}
Here, $t_n:k \longrightarrow N$ denotes the map that associates $r\in k$ to $rn\in N$. Using the $A$-linearity of $\varphi$, it is starightforward to check that $\tilde{\varphi}(n)$ indeed lies in ${_A}\mathfrak S(\mathcal M,\mathcal N)$. Next, we check that $\tilde{\varphi}:N \longrightarrow {_A}\mathfrak S(\mathcal M,\mathcal N)$ is $H^*$-linear. For any $h^* \in H^*$ and $n \in N$, we have
\begin{align*}
&\big(\mu_{ {_A}\mathfrak S(\mathcal M,\mathcal N)}^{H^*} \circ (H^* \otimes \tilde{\varphi})\big)(h^* \otimes n)=\mu_{ {_A}\mathfrak S(\mathcal M,\mathcal N)}^{H^*}(h^* \otimes \tilde{\varphi}(n))=h^* \cdot \tilde{\varphi}(n)=h^* \cdot \big(\varphi \circ (\mathcal M \otimes t_n)\big)\\
&\quad =(\mathcal N \otimes h^*)(\mathcal N \otimes \mu_{H^{op}})(\Delta^H_\mathcal N \otimes H)\big(\left(\varphi \circ (\mathcal M \otimes t_n)\right) \otimes H\big)(\mathcal M \otimes S^{-1})\Delta_\mathcal M^H\\
&\quad =(\mathcal N \otimes h^*)(\mathcal N \otimes \mu_{H^{op}})(\Delta^H_\mathcal N \otimes H)(\varphi \otimes H)(\mathcal M \otimes t_n \otimes H)(\mathcal M \otimes S^{-1})\Delta_\mathcal M^H\\
&\quad =(\mathcal N \otimes h^*)(\mathcal N \otimes \mu_{H^{op}})(\varphi \otimes H \otimes H)(\Delta^H_{\mathcal M \otimes N} \otimes H)(\mathcal M \otimes t_n \otimes H)(\mathcal M \otimes S^{-1})\Delta_\mathcal M^H \\
&\quad =(\mathcal N \otimes h^*)(\mathcal N \otimes \mu_{H^{op}})(\varphi \otimes H \otimes H)(\mathcal M \otimes N \otimes \mu_H \otimes H)(\mathcal M \otimes T_{H,N} \otimes H \otimes H)(\Delta_\mathcal M^H \otimes \Delta^H_N \otimes H)(\mathcal M \otimes t_n \otimes H)(\mathcal M \otimes S^{-1})\Delta_\mathcal M^H\\\
&\quad =(\mathcal N \otimes h^*)(\mathcal N \otimes \mu_{H^{op}})(\varphi \otimes H \otimes H)(\mathcal M \otimes N \otimes \mu_H \otimes H)(\mathcal M \otimes T_{H,N} \otimes H \otimes H)(\mathcal M \otimes H \otimes \Delta_N^H \otimes H)(\Delta_\mathcal M^H \otimes N \otimes H) \\
& \qquad (\mathcal M \otimes t_n \otimes H)(\mathcal M \otimes S^{-1})\Delta_\mathcal M^H\\
&\quad =(\mathcal N \otimes h^*)(\mathcal N \otimes \mu_{H^{op}})(\varphi \otimes H \otimes H)(\mathcal M \otimes N \otimes \mu_H \otimes H)(\mathcal M \otimes T_{H,N} \otimes H \otimes H)(\mathcal M \otimes H \otimes \Delta_N^H \otimes H)(\mathcal M \otimes H \otimes t_n \otimes H)\\
& \qquad (\mathcal M \otimes H \otimes S^{-1})(\Delta_\mathcal M^H \otimes H)\Delta_\mathcal M^H\\
&\quad =(\mathcal N \otimes h^*)(\mathcal N \otimes \mu_{H^{op}})(\varphi \otimes H \otimes H)(\mathcal M \otimes N \otimes \mu_H \otimes H)(\mathcal M \otimes T_{H,N} \otimes H \otimes H)(\mathcal M \otimes H \otimes \Delta_N^H \otimes H)(\mathcal M \otimes H \otimes t_n \otimes H)\\
& \qquad (\mathcal M \otimes H \otimes S^{-1})(\mathcal M \otimes \Delta_H)\Delta_\mathcal M^H \\
& \quad =\varphi(\mathcal M \otimes N \otimes h^*)(\mathcal M \otimes N \otimes \mu_{H^{op}})(\mathcal M \otimes N \otimes \mu_H \otimes H)(\mathcal M \otimes T_{H,N} \otimes H \otimes H)(\mathcal M \otimes H \otimes \Delta_N^H \otimes H)(\mathcal M \otimes H \otimes N \otimes S^{-1})\\
&\qquad (\mathcal M \otimes H \otimes t_n \otimes H )(\mathcal M \otimes \Delta_H)\Delta_\mathcal M^H\\
& \quad =\varphi(\mathcal M \otimes N \otimes h^*)(\mathcal M \otimes N \otimes \mu_{H^{op}})(\mathcal M \otimes N \otimes \mu_H \otimes H)(\mathcal M \otimes T_{H,N} \otimes H \otimes H)(\mathcal M \otimes H \otimes N \otimes H \otimes S^{-1}) (\mathcal M \otimes H \otimes \Delta_N^H \otimes H)\\
&\qquad(\mathcal M \otimes H \otimes T_{H,N}) (\mathcal M \otimes \Delta_H \otimes N)(\Delta_\mathcal M^H \otimes N)(\mathcal M \otimes t_n)\\
& =\varphi(\mathcal M \otimes N \otimes h^*)(\mathcal M \otimes N\otimes \mu_{H^{op}})(\mathcal M \otimes N \otimes \mu_H \otimes H)(\mathcal M \otimes N \otimes H \otimes H \otimes S^{-1})(\mathcal M \otimes T_{H,N} \otimes H \otimes H)(\mathcal M \otimes H \otimes \Delta_N^H \otimes H)\\
&\qquad (\mathcal M \otimes H \otimes T_{H,N}) (\mathcal M \otimes \Delta_H \otimes N)(\Delta_\mathcal M^H \otimes N)(\mathcal M \otimes t_n)\\
& =\varphi(\mathcal M \otimes N \otimes h^*)(\mathcal M \otimes N \otimes \mu_{H^{op}})(\mathcal M \otimes N \otimes H \otimes S^{-1})(\mathcal M \otimes N \otimes \mu_H \otimes H)(\mathcal M \otimes T_{H,N} \otimes H \otimes H)(\mathcal M \otimes H \otimes \Delta_N^H \otimes H)\\
&\qquad (\mathcal M \otimes H \otimes T_{H,N}) (\mathcal M \otimes \Delta_H \otimes N)(\Delta_\mathcal M^H \otimes N)(\mathcal M \otimes t_n)
\end{align*}
\begin{align*}
& =\varphi(\mathcal M \otimes N \otimes h^*)(\mathcal M \otimes N \otimes \mu_{H^{op}})(\mathcal M \otimes N \otimes H \otimes S^{-1})(\mathcal M \otimes N \otimes \mu_H \otimes H)(\mathcal M \otimes N \otimes T_{H,H} \otimes H)(\mathcal M \otimes N \otimes H \otimes \Delta_H)\\
& \qquad (\mathcal M \otimes \Delta_N^H \otimes H)(\mathcal M \otimes T_{H,N})(\Delta_\mathcal M^H \otimes N)(\mathcal M \otimes t_n)\\
& =\varphi(\mathcal M \otimes N \otimes h^*)(\mathcal M \otimes N \otimes H \otimes \varepsilon_H)(\mathcal M \otimes \Delta_N^H \otimes H)(\mathcal M \otimes T_{H,N})(\Delta_\mathcal M^H \otimes N)(\mathcal M \otimes t_n)
\end{align*}
where the last equality follows from the fact that $\mu_{H^{op}}( H \otimes S^{-1})( \mu_H \otimes H)( T_{H,H} \otimes H)( H \otimes \Delta_H)=H \otimes \varepsilon_H$. We now have
\begin{align*}
\big(\mu_{ {_A}\mathfrak S(\mathcal M,\mathcal N)}^{H^*} \circ (H^* \otimes \tilde{\varphi})\big)(h^* \otimes n)& =\varphi(\mathcal M \otimes N \otimes h^*)(\mathcal M \otimes \Delta_N^H)(\mathcal{M} \otimes \varepsilon_H \otimes N )(\Delta_\mathcal M^H \otimes N)(\mathcal M \otimes t_n)\\
&=\varphi(\mathcal M \otimes N \otimes h^*)(\mathcal M \otimes \Delta_N^H)(\mathcal M \otimes t_n)=(\tilde{\varphi} \circ \mu^{H^*}_N)(h^* \otimes n)
\end{align*}
 Conversely, suppose that $g:N \longrightarrow {_A}\mathfrak S(\mathcal M,\mathcal N)$ is an $H^*$-linear morphism. We define $\bar{g}:\mathcal M \otimes N \longrightarrow \mathcal N$ to be the following composition
\begin{equation*}
\mathcal M \otimes N \xrightarrow{\mathcal M \otimes g} \mathcal M \otimes ({_A}\mathfrak S(\mathcal M,\mathcal N)) \xrightarrow{ev} \mathcal N
\end{equation*}
where $ev$ denotes the evaluation map. It may be verified directly that $\bar{g}\in {_A}\mathfrak S^H(\mathcal M \otimes N, \mathcal N)$ and these two associations are inverse to each other.
\end{proof}

\begin{thm}
Let $(\mathcal M, \mu^A_\mathcal M, \Delta^H_\mathcal M)$ and $(\mathcal N, \mu^A_\mathcal N, \Delta^H_\mathcal N)$ be objects in ${_A}\mathfrak S^H$. Suppose that $(\mathcal M, \mu^A_\mathcal M)$ is finitely generated in ${_A}\mathfrak S$. Then, $ {_{{_A}{\mathfrak S}}}HOM(\mathcal M,\mathcal N) \cong  {_A}\mathfrak S(\mathcal M,\mathcal N)$ as $H$-comodules.
\end{thm}
\begin{proof}
 For any $\phi \in {_A}\mathfrak S(\mathcal M,\mathcal N)$, we know from Proposition \ref{6.4fd} that the morphism $\rho(\phi)$ as defined in \eqref{rho} lies in ${_A}\mathfrak S(\mathcal M, \mathcal N \otimes H)$. Since $\mathcal M$ is finitely generated in ${_A}\mathfrak S$ and the $A$-action on $\mathcal N \otimes H$ is determined by the $A$-action on $\mathcal N$, we have  the induced linear map that we continue to denote by $\rho$:
\begin{equation*}
\rho:{_A}\mathfrak S(\mathcal M, \mathcal N) \longrightarrow {_A}\mathfrak S(\mathcal M, \mathcal N \otimes H) \cong  {_A}\mathfrak S(\mathcal M, \mathcal N) \otimes H \qquad \phi \mapsto \rho(\phi)
\end{equation*} 
From the left $H^*$-module action on  ${_A}\mathfrak S(\mathcal M, \mathcal N)$ described in \eqref{6.6yc}, we see that the structure map
${_A}\mathfrak S(\mathcal M, \mathcal N)  \longrightarrow Hom_k\left(H^*, {_A}\mathfrak S(\mathcal M, \mathcal N)\right)$
factors through ${_A}\mathfrak S(\mathcal M, \mathcal N) \otimes H$. Therefore, ${_A}\mathfrak S(\mathcal M, \mathcal N)$ is a rational $H^*$-module and hence, using Proposition \ref{6.6i} we have
\begin{equation*}
{_{{_A}{\mathfrak S}}}HOM(\mathcal M,\mathcal N) \cong  {_{H^*}}R^H({_{{_A}{\mathfrak S}}}(\mathcal M,\mathcal N)) =   {_A}\mathfrak S(\mathcal M,\mathcal N)
\end{equation*}
This proves the result.
\end{proof}

\section{Coinduction functor and injective envelopes}

We recall that a left integral on a Hopf algebra $H$ is a left $H$-comodule map $\phi: H\longrightarrow k$ which means that $\phi$ satisfies $h_1\phi(h_2)=\phi(h)1_H$ for all $h \in H$. Similarly, a right integral is a right $H$-comodule map $\phi: H\longrightarrow k$, i.e. $\phi(h_1)h_2=\phi(h)1_H$.
A Hopf algebra $H$ is said to be cosemisimple if $H$ is equipped with a left integral $\phi$ satisfying $\phi(1_H)=1$. Moreover, in a cosemisimple Hopf algebra, a left integral is also a right integral and vice-versa (see, for instance \cite[p. 226]{DNR}). In this section, we will always assume $H$ to be cosemisimple.

\smallskip
For an object $(\mathcal N,\Delta_{\mathcal N}^H)\in \mathfrak S^H$, we now consider the following equalizer in $\mathfrak S$
\begin{equation}\label{7.1eq}
\mathfrak C^H(\mathcal N):=Eq\left( \mathcal N\doublerightarrow{\qquad(\mathcal N\otimes u_H)\qquad }{ \Delta^H_{\mathcal N}}\mathcal N\otimes H\right)
\end{equation} From \eqref{7.1eq} it is immediate that the composition $\mathfrak C^H(\mathcal N)\hookrightarrow  \mathcal N\xrightarrow{\Delta^H_{\mathcal N}}\mathcal N
\otimes H$ factors through $\mathfrak C^H(\mathcal N)\otimes H$. This makes $\mathfrak C^H(\mathcal N)$ into an $H$-comodule subobject 
of $\mathcal N$ in $\mathfrak S$.  We now define the following two full subcategories of $\mathfrak S^H$:
\begin{equation}\label{tor7}
 \mathcal T(\mathfrak S^H):=\{\mbox{$\mathcal N\in \mathfrak S^H$ $\vert$ $\mathfrak C^H(\mathcal N)=\mathcal N$ }\} \qquad \mathcal F(\mathfrak S^H):=\{\mbox{$\mathcal N\in \mathfrak S^H$ $\vert$ $\mathfrak C^H(\mathcal N)=0$ }\} 
\end{equation}
\begin{lem}\label{L7.01}
The subcategory $ \mathcal T(\mathfrak S^H)$ is closed under subobjects, quotients and direct sums. The subcategory  $ \mathcal F(\mathfrak S^H)$ is closed under subobjects and direct sums.
\end{lem}
\begin{proof}
By \eqref{7.1eq}, we know that $\mathfrak C^H(\mathcal N)=Ker(\zeta(\mathcal N))$ for any $\mathcal N\in \mathfrak S^H$, where $\zeta(\mathcal N)=(\mathcal N\otimes u_H)-\Delta^H_{\mathcal N}$. From this, it is immediate that $ \mathcal T(\mathfrak S^H)$ and $ \mathcal F(\mathfrak S^H)$ are closed under direct sums. We consider $\mathcal N'\subseteq \mathcal N$ in $\mathfrak S^H$.  If $\mathcal N\in \mathcal T(\mathfrak S^H)$,  then $\zeta(\mathcal N)=0$ and hence so is its restriction $\zeta(\mathcal N')$.  If $\mathcal N\in \mathcal F(\mathfrak S^H)$, we get $\mathfrak C^H(\mathcal N')\hookrightarrow \mathfrak C^H(\mathcal N)=0$.

\smallskip
On the other hand, let $\mathcal N\twoheadrightarrow \mathcal N''$ be a quotient in $\mathfrak S^H$.    If $\mathcal N\in \mathcal T(\mathfrak S^H)$,  then $\zeta(\mathcal N)=0$ and hence so is   $\zeta(\mathcal N'')=0$ on the quotient $\mathcal N''$. 
\end{proof}
 For any $\mathcal N\in \mathfrak S^H$, we now consider the following morphism
\begin{equation}\label{7.2eq}
\pi_{\mathcal N}^\phi: \mathcal N \xrightarrow{\Delta_{\mathcal N}^H}\mathcal N\otimes H \xrightarrow{\mathcal N\otimes \phi}\mathcal N
\end{equation} Since $\phi:H\longrightarrow k$ is a right integral, i.e., $\phi$ is a morphism of right $H$-comodules, we note that 
\eqref{7.2eq} becomes a morphism in $\mathfrak S^H$. We now set $\mathfrak K^H(\mathcal N):=Ker(\pi^{\phi}_{\mathcal N})$ in 
$\mathfrak S^H$.
\begin{lem}\label{L7.1t} Let $ (\mathcal N,\Delta_{\mathcal N}^H)\in \mathfrak S^H$. Then, the image of $\pi^\phi_{\mathcal N}$ lies in $\mathfrak C^H(\mathcal N)$. 
\end{lem}
\begin{proof}
Since $\phi:H\longrightarrow k$ is a left integral, we know that $(H\otimes \phi)\circ \Delta_H=u_H\circ \phi$. Accordingly, we have
\begin{equation*}
\begin{array}{ll}
\Delta^H_{\mathcal N}\circ \pi_{\mathcal N}^\phi&= \Delta^H_{\mathcal N}\circ (\mathcal N\otimes \phi)\circ \Delta_{\mathcal N}^H\\
&= (\mathcal N\otimes H\otimes \phi)\circ (\Delta^H_{\mathcal N}\otimes H)\circ \Delta^H_{\mathcal N}\\
&= (\mathcal N\otimes H\otimes \phi)\circ (\mathcal N\otimes \Delta_H)\circ \Delta^H_{\mathcal N} \\
& = (\mathcal N\otimes u_H)\circ (\mathcal N\otimes \phi)\circ \Delta^H_{\mathcal N}=(\mathcal N\otimes u_H)\circ \pi^\phi_{\mathcal N}\\
\end{array}
\end{equation*} From the equalizer in \eqref{7.1eq}, it is now clear that $\pi^\phi_{\mathcal N}$ factors through $\mathfrak C^H(\mathcal N)$. 
\end{proof}

By abuse of notation, we continue to denote by $\pi^\phi_{\mathcal N}:\mathcal N\longrightarrow \mathfrak C^H(\mathcal N)$ the canonical morphism induced by  $\pi^\phi_{\mathcal N}$.

\begin{thm}\label{P7.2}
 Let $ (\mathcal N,\Delta_{\mathcal N}^H)\in \mathfrak S^H$. Then, $\pi^\phi_{\mathcal N}:\mathcal N\longrightarrow \mathfrak C^H(\mathcal N)$ is a split epimorphism which leads to a decomposition $\mathcal N=\mathfrak C^H(\mathcal N)\oplus \mathfrak K^H(\mathcal N)$  in $\mathfrak S^H$. 
\end{thm}

\begin{proof}
We denote by $\iota_{\mathcal N}:\mathfrak C^H(\mathcal N)\hookrightarrow \mathcal N$ the canonical inclusion. Using the fact that $\phi(1_H)=1$, it is easy to see that the composition
$\pi^\phi_{\mathcal N}\circ \iota_{\mathcal N}=id_{\mathfrak C^H(\mathcal N)}$. Accordingly, $\pi^\phi_{\mathcal N}:\mathcal N\longrightarrow 
\mathfrak C^H(\mathcal N)$ is a split epimorphism.
\end{proof}

\begin{thm}\label{P7.5f}
The subcategories $(\mathcal T(\mathfrak S^H),\mathcal F(\mathfrak S^H))$ detemine a torsion theory on $\mathfrak S^H$. 
\end{thm}

\begin{proof}
By Proposition \ref{P7.2}, every $\mathcal N\in \mathfrak S^H$ fits into a (split) short exact sequence $0\longrightarrow \mathfrak C^H(\mathcal N)\longrightarrow 
\mathcal N\longrightarrow \mathfrak K^H(\mathcal N)\longrightarrow 0$. From the definitions in \eqref{7.1eq} and \eqref{tor7}, it is clear that
$\mathfrak C^H(\mathcal N)\in \mathcal T(\mathfrak S^H)$. We denote by $\iota'_{\mathcal N}$ the canonical inclusion $\iota'_{\mathcal N}:
\mathfrak K^H(\mathcal N)=Ker(\pi_{\mathcal N}^\phi)\hookrightarrow \mathcal N$. 
We claim that $\mathfrak K^H(\mathcal N)\in \mathcal F(\mathfrak S^H)$. For this, we suppose there is a subobject $\iota'':\mathcal M\hookrightarrow
\mathfrak K^H(\mathcal N)$ such that $\Delta_{\mathfrak K^H(\mathcal N)}\circ \iota''=(\mathfrak K^H(\mathcal N)\otimes u_H)\circ \iota''=(\iota''\otimes H)\circ (\mathcal M\otimes u_H)$. Since $\phi(1_H)=1$, we now have
\begin{equation}
\begin{array}{ll}
\iota'_{\mathcal N}\circ \iota''&= (\mathcal N\otimes \phi)\circ (\iota'_{\mathcal N}\otimes H)\circ (\iota''\otimes H)\circ (\mathcal M\otimes u_H)\\
&= (\mathcal N\otimes \phi)\circ (\iota'_{\mathcal N}\otimes H)\circ ((\mathfrak K^H(\mathcal N)\otimes u_H))\circ\iota''\\
&= (\mathcal N\otimes \phi)\circ (\iota'_{\mathcal N}\otimes H)\circ \Delta_{\mathfrak K^H(\mathcal N)}\circ \iota''\\
&= (\mathcal N\otimes \phi)\circ \Delta_{\mathcal N}^H\circ \iota'_{\mathcal N}\circ \iota''=0\\
\end{array}
\end{equation} which gives $\mathcal M=0$. 

\smallskip Now suppose that $\psi:\mathcal N'\longrightarrow \mathcal N''$ is a morphism in $\mathfrak S^H$ with $\mathcal N'\in \mathcal T(\mathfrak S^H)$ and $\mathcal N''\in \mathcal F(\mathfrak S^H)$. By Lemma \ref{L7.01}, the quotient $Im(\psi)$ of $\mathcal N'$ lies in $\mathcal T(\mathfrak S^H)$. Also, $Im(\psi)\subseteq \mathcal N''\in \mathcal F(\mathfrak S^H)$ lies in 
$\mathcal F(\mathfrak S^H)$. From \eqref{tor7}, it is evident that $\mathcal T(\mathfrak S^H)\cap \mathcal F(\mathfrak S^H)=0$ and hence $\psi=0$. It now follows that $(\mathcal T(\mathfrak S^H),\mathcal F(\mathfrak S^H))$ detemines a torsion theory on the abelian category $\mathfrak S^H$ (see, for instance, \cite[$\S$ I.1]{BelR}). 
\end{proof}

\begin{lem}\label{quo7}
The subcategory $ \mathcal F(\mathfrak S^H)$ is closed under quotients.
\end{lem}

\begin{proof}
Let $\xi:\mathcal N\twoheadrightarrow \mathcal N''$ be a quotient in $\mathfrak S^H$.  Suppose that $\mathcal N\in \mathcal F(\mathfrak S^H)$, i.e.,
$\mathfrak C^H(\mathcal N)=Im(\pi^{\phi}_{\mathcal N})=0$, or $\pi_{\mathcal N}^\phi=0$.  Since $\pi^\phi_{\mathcal N''}\circ \xi=
\xi\circ \pi^\phi_{\mathcal N}=0$ and $\xi$ is an epimorphism, we get $\pi^\phi_{\mathcal N''}=0$. Hence, $\mathcal N''\in \mathcal F(\mathfrak S^H)$.
\end{proof}

\begin{cor}\label{quocor7}
The subcategories $(\mathcal F(\mathfrak S^H),\mathcal T(\mathfrak S^H))$ detemine a torsion theory on $\mathfrak S^H$. 
\end{cor}

\begin{proof}
For every $\mathcal N\in \mathfrak S^H$, the direct sum decomposition in Proposition \ref{P7.2} also gives a short exact sequence $0\longrightarrow \mathfrak K^H(\mathcal N)\longrightarrow 
\mathcal N\longrightarrow \mathfrak C^H(\mathcal N)\longrightarrow 0$. From the proof of Proposition \ref{P7.5f}, we already know that
$\mathfrak C^H(\mathcal N)\in \mathcal T(\mathfrak S^H)$ and $\mathfrak K^H(\mathcal N)\in \mathcal F(\mathfrak S^H)$. 

\smallskip Now let  $\psi:\mathcal N'\longrightarrow \mathcal N''$ be a morphism in $\mathfrak S^H$ with $\mathcal N'\in \mathcal F(\mathfrak S^H)$ and $\mathcal N''\in \mathcal T(\mathfrak S^H)$. By Lemma \ref{quo7}, the quotient $Im(\psi)$ of $\mathcal N'$ lies in $\mathcal F(\mathfrak S^H)$. By Lemma 
\ref{L7.01}, $Im(\psi)\subseteq \mathcal N''\in \mathcal T(\mathfrak S^H)$ lies in 
$\mathcal T(\mathfrak S^H)$. Again since $\mathcal T(\mathfrak S^H)\cap \mathcal F(\mathfrak S^H)=0$, we get $\psi=0$. 
\end{proof}

\begin{lem}
\label{L7.4b} The functor $\mathcal N\mapsto \mathfrak C^H(\mathcal N)$ is exact, and preserves all colimits.
\end{lem}

\begin{proof}
 Since $(\mathcal T(\mathfrak S^H),\mathcal F(\mathfrak S^H))$ is a torsion theory, we can split a morphism $\psi:\mathcal M\longrightarrow \mathcal N$   in $\mathfrak S^H$ into its components $\mathfrak C^H(\psi):\mathfrak C^H(\mathcal M)\longrightarrow \mathfrak C^H(\mathcal N)$ and 
$\mathfrak K^H(\psi):\mathfrak K^H(\mathcal M)\longrightarrow \mathfrak K^H(\mathcal N)$. Then, we have $Ker(\psi)=Ker(\mathfrak C^H(\psi))\oplus Ker(\mathfrak K^H(\psi))$ and $Coker(\psi)=Coker(\mathfrak C^H(\psi))\oplus Coker(\mathfrak K^H(\psi))$. By Lemma \ref{L7.01} and Lemma \ref{quo7}, we see that $Ker(\mathfrak C^H(\psi))$, $Coker(\mathfrak C^H(\psi))\in  \mathcal T(\mathfrak S^H)$ and $Ker(\mathfrak K^H(\psi))$, $Coker(\mathfrak K^H(\psi))\in  \mathcal F(\mathfrak S^H)$. Accordingly, we have $\mathfrak C^H(Coker(\psi))=Coker(\mathfrak C^H(\psi))$ and $\mathfrak C^H(Ker(\psi))=Ker(\mathfrak C^H(\psi))$. This shows that the functor $\mathcal N\mapsto \mathfrak C^H(\mathcal N)$ for $\mathcal N\in \mathfrak S^H$ is exact. 

\smallskip
Since filtered colimits commute with finite limits in $\mathfrak S^H$, it follows from \eqref{7.1eq} that the functor $\mathcal N\mapsto \mathfrak C^H(\mathcal N)$ preserves direct sums. Since every colimit may be expressed as a combination of cokernels and direct sums, we see that  $\mathcal N\mapsto \mathfrak C^H(\mathcal N)$ preserves all colimits.
\end{proof}

We now let $(A,\Delta_A^H:A\longrightarrow A\otimes H)$ be a  right $H$-comodule algebra as in Section 4 and consider the category ${_A}\mathfrak S^H$ of relative $(A,H)$-Hopf module objects in $\mathfrak S$. We also consider the subalgebra of coinvariants
\begin{equation}\label{coina}
B:=\{\mbox{$a\in A$ $\vert$ $\Delta_A^H(a)=a\otimes 1_H$}\}
\end{equation} and denote by $i:B\longrightarrow A$ the canonical inclusion.

\begin{lem}\label{L7.7vq}
Let $(\mathcal M,\mu_{\mathcal M}^A,\Delta_{\mathcal M}^H)\in {_A}\mathfrak S^H$. Then, $\mathfrak C^H(\mathcal M)$ is canonically equipped with the structure of a left $B$-module object
in $\mathfrak S$.
\end{lem}

\begin{proof}
Given the canonical inclusions,  $\iota_{\mathcal M}:\mathfrak C^H(\mathcal M)\hookrightarrow \mathcal M$ and $i:B\hookrightarrow A$, we will construct a structure map for $\mathfrak C^H(\mathcal M)\in {_B}\mathfrak S$ by showing that the composition $B\otimes  \mathfrak C^H(\mathcal M)\xrightarrow{i\otimes \iota_{\mathcal M}}A\otimes\mathcal M\xrightarrow{\mu_{\mathcal M}^A}\mathcal M$ factors through
the equalizer $\mathfrak C^H(\mathcal M)$. Using the notation of Definition \ref{D5.1cs}, we now note that
\begin{equation*}
\begin{array}{ll}
\Delta_{\mathcal M}^H\circ \mu_{\mathcal M}^A\circ (i\otimes \iota_{\mathcal M}) & =(\mu^A_{\mathcal M}\otimes \mu_H)\circ(A\otimes T_{H,\mathcal M}\otimes H)\circ (\Delta_A^H \otimes \Delta^H_{\mathcal M})\circ (i\otimes \iota_{\mathcal M}) \\
& =(\mu^A_{\mathcal M}\otimes \mu_H)\circ(A\otimes T_{H,\mathcal M}\otimes H)\circ (A\otimes u_H\otimes \mathcal M\otimes u_H)\circ (i\otimes \iota_{\mathcal M}) \\
&= (\mathcal M\otimes u_H)\circ \mu_{\mathcal M}^A\circ (i\otimes \iota_{\mathcal M}) \\
\end{array}
\end{equation*} From the expression for $\mathfrak C^H(\mathcal M)$ in \eqref{7.1eq}, the result is now clear.
\end{proof}

\begin{thm}
\label{P7.8vk} Let $A$ be a right $H$-comodule algebra and let $B\subseteq A$ the subalgebra of coinvariants of $A$. Then, the functor ${_A}\mathfrak S^H\longrightarrow 
{_B}\mathfrak S$ given by $\mathcal M\mapsto \mathfrak C^H(\mathcal M)$ has a right adjoint ${_B}HOM(A,\_\_):{_B}\mathfrak S\longrightarrow {_A}\mathfrak S^H$. In other words, we have natural isomorphisms
\begin{equation}
{_B}\mathfrak S(\mathfrak C^H(\mathcal M),\mathcal N)\cong {_A}\mathfrak S^H(\mathcal M, {_B}HOM(A,\mathcal N))
\end{equation} for $\mathcal M\in {_A}\mathfrak S^H$ and $\mathcal N\in {_B}\mathfrak S$. 
\end{thm}

\begin{proof}
We know that colimits in ${_A}\mathfrak S^H$ and ${_B}\mathfrak S$ are computed in $\mathfrak S$. Accordingly, it follows from Lemma \ref{L7.4b} that the functor ${_A}\mathfrak S^H\longrightarrow 
{_B}\mathfrak S$ given by $\mathcal M\mapsto \mathfrak C^H(\mathcal M)$ preserves colimits. By Theorem \ref{Thm5.3sq}, we know that ${_A}\mathfrak S^H$ is a Grothendieck category
and so is ${_B}\mathfrak S$. It now follows from \cite[Proposition 8.3.27]{KS} that the colimit preserving functor given by $\mathcal M\mapsto 
\mathfrak C^H(\mathcal M)$  has a right adjoint ${_B}HOM(A,\_\_):{_B}\mathfrak S\longrightarrow {_A}\mathfrak S^H$. 
\end{proof}

\begin{cor}\label{C7.9} The functor ${_B}HOM(A,\_\_):{_B}\mathfrak S\longrightarrow {_A}\mathfrak S^H$ preserves injectives.

\end{cor}
\begin{proof}   From Lemma \ref{L7.4b} it is clear that the left adjoint  ${_A}\mathfrak S^H\longrightarrow 
{_B}\mathfrak S$ given by $\mathcal M\mapsto \mathfrak C^H(\mathcal M)$ is exact. Hence, the right adjoint ${_B}HOM(A,\_\_):{_B}\mathfrak S\longrightarrow {_A}\mathfrak S^H$ preserves injectives.

\end{proof}

Since $A$ is an $H$-comodule algebra, it is easily verified that given $\mathcal M\in {_B}\mathfrak S$, we must have $A\otimes_B\mathcal M\in {_A}\mathfrak S^H$ with $H$-coaction extended from $A$.

\begin{lem}\label{L7.9cj}
Let $H$ be a cosemisimple Hopf algebra. Then, for any $\mathcal M\in {_B}\mathfrak S$,  we have $\mathfrak C^H(A\otimes_B\mathcal M)\cong \mathcal M$
as objects of ${_B}\mathfrak S$.
\end{lem}

\begin{proof}
Since $H$ is cosemisimple, using the left integral $\phi:H\longrightarrow k$, we consider  the morphism
$\pi_A^\phi:A\longrightarrow A^{co H}=B$ given by $a\mapsto \sum \phi(a_1)a_0$ where $\Delta_A^H(a)=\sum a_0\otimes a_1$ for any $a\in A$.  We note that this map $\pi^\phi_A:A\longrightarrow A^{co H}$ is a morphism of $B$-bimodules and that the composition with the inclusion $B\hookrightarrow A$ gives the identity on $B$. Accordingly, we have a  split $A\cong B\oplus B'$ as   $B$-bimodules. Then for $\mathcal M\in {_B}\mathfrak S$,  we have  
$(A\otimes_B\mathcal M)\cong (B\otimes_B\mathcal M)\oplus (B'\otimes_B\mathcal M)$, which shows that there is an inclusion $i\otimes_B\mathcal M:\mathcal M=(B\otimes_B\mathcal M)\hookrightarrow (A\otimes_B\mathcal M)$ in ${_B}\mathfrak S$. 

\smallskip
We now consider the composition $A\xrightarrow{\pi_A^\phi}B\xrightarrow{i}A$, where $i:B\hookrightarrow A$ is the canonical inclusion. Since the $H$-coaction on $(A\otimes_B\mathcal M)$ is induced from the $H$-coaction on $A$, we note that 
\begin{equation}
\begin{CD}
\pi_{(A\otimes_B\mathcal M)}^\phi=((i\circ \pi^\phi_A)\otimes_B\mathcal M): (A\otimes_B\mathcal M)@>\Delta_{(A\otimes_B\mathcal M)}^H>> (A\otimes_B\mathcal M)\otimes H @>(A\otimes_B\mathcal M)\otimes \phi>> (A\otimes_B\mathcal M)
\end{CD}
\end{equation} in the notation of \eqref{7.2eq}. By Proposition \ref{P7.2}, we know that $\mathfrak C^H(A\otimes_B\mathcal M)$ is the image of the morphism $\pi_{(A\otimes_B\mathcal M)}^\phi=((i\circ \pi^\phi_A)\otimes_B\mathcal M)$. Now since $\pi^\phi_A:A\longrightarrow B$ is an epimorphism, so is $\pi^\phi_A\otimes_B\mathcal M:(A\otimes_B\mathcal M)\longrightarrow (B\otimes_B\mathcal M)=\mathcal M$. We have already shown that $i\otimes_B\mathcal M:\mathcal M\longrightarrow (A\otimes_B\mathcal M)$ is a monomorphism. It now follows that $\mathcal M$ is the image of $\pi_{(A\otimes_B\mathcal M)}^\phi$, which gives $\mathfrak C^H(A\otimes_B\mathcal M)=\mathcal M$.
\end{proof}

\begin{thm}\label{P7.10gs}
Let $A$ be a right $H$-comodule algebra and $B=A^{co H}$. Then, we have natural isomorphisms
\begin{equation}\label{7.7yt}
{_A}\mathfrak S^H(A\otimes_B\mathcal M,\mathcal N)\cong {_B}\mathfrak S(\mathcal M,\mathfrak C^H(\mathcal N)) 
\end{equation}
for $\mathcal M\in {_B}\mathfrak S$ and $\mathcal N\in {_A}\mathfrak S^H$.
\end{thm}
\begin{proof}
Let $\psi\in {_A}\mathfrak S^H(A\otimes_B\mathcal M,\mathcal N)$. Then, $\psi\in {_A}\mathfrak S^H(A\otimes_B\mathcal M,\mathcal N)$ corresponds to 
$\psi'\in {_B}\mathfrak S(\mathcal M,\mathcal N)$ given by $\psi':\mathcal M\xrightarrow{(u_A\otimes \mathcal M)}A\otimes_B\mathcal M
\xrightarrow{\psi}\mathcal N$. We know that $\Delta_A^H\circ u_A=u_A\otimes u_H$. From the $H$-colinearity of $\psi$ and the fact that
$\psi'=\psi \circ (u_A\otimes \mathcal M)$, we now have
\begin{equation}
\Delta_{\mathcal N}^H\circ \psi'=(\psi'\otimes H)\circ (\mathcal M\otimes u_H)= (\mathcal N\otimes u_H)\circ \psi' 
\end{equation}  From the equalizer in \eqref{7.1eq}, we now see that $\psi'$ factors through $\mathfrak C^H(\mathcal N)$.  These arguments can be easily reversed, thus showing the isomorphism in \eqref{7.7yt}.
\end{proof}

\begin{thm}\label{P7.11sgs}
Let $A$ be a right $H$-comodule algebra and $B\subseteq A$ the subalgebra of coinvariants of $A$. Then, for any $\mathcal M\in {_B}\mathfrak S$, we have a canonical isomorphism
\begin{equation}\label{iso7.9}
\mathfrak C^H({_B}HOM(A,\mathcal M))\cong \mathcal M
\end{equation}
\end{thm}

\begin{proof}
We consider some $\mathcal M'\in {_B}\mathfrak S$. Applying Proposition \ref{P7.8vk}, Lemma \ref{L7.9cj} and Proposition \ref{P7.10gs}, we have isomorphisms
\begin{equation}
\begin{array}{ll}
{_B}\mathfrak S(\mathcal M',\mathfrak C^H({_B}HOM(A,\mathcal M)))&\cong {_A}\mathfrak S^H(A\otimes_B\mathcal M',{_B}HOM(A,\mathcal M))\\
&\cong {_B}\mathfrak S(\mathfrak C^H(A\otimes_B\mathcal M'),\mathcal M)\\
&\cong {_B}\mathfrak S(\mathcal M',\mathcal M)\\
\end{array}
\end{equation} This proves the result.
\end{proof}

\begin{lem}\label{L7.12aq} (a) Let $\mathcal N\in {_B}\mathfrak S$ and let $\mathcal M\subseteq {_B}HOM(A,\mathcal N)$ be a subobject in ${_A}\mathfrak S^H$. Then, $\mathfrak C^H(\mathcal M)=0$ implies $\mathcal M=0$.

\smallskip
(b)  The functor ${_B}HOM(A,\_\_):{_B}\mathfrak S\longrightarrow {_A}\mathfrak S^H$ preserves essential monomorphisms. 

\end{lem}

\begin{proof} (a) If $\mathfrak C^H(\mathcal M)=0$,  we have by Proposition \ref{P7.8vk} that $0={_B}\mathfrak S(\mathfrak C^H(\mathcal M),\mathcal N)\cong {_A}\mathfrak S^H(\mathcal M, {_B}HOM(A,\mathcal N))$.  Since $\mathcal M\subseteq {_B}HOM(A,\mathcal N)$, this  is a contradiction unless
$\mathcal M=0$.

\smallskip
(b) Let $\mathcal N'\hookrightarrow\mathcal N$ be an essential monomorphism in ${_B}\mathfrak S$. Because ${_B}HOM(A,\_\_):{_B}\mathfrak S\longrightarrow {_A}\mathfrak S^H$ is a right adjoint, it preserves monomorphisms. We now consider $\mathcal M\subseteq {_B}HOM(A,\mathcal N)$ such that 
$\mathcal M\cap {_B}HOM(A,\mathcal N')=0$. From Proposition \ref{P7.10gs}, we know that $\mathfrak C^H$ is also a right adjoint and preserves limits. Accordingly,
we get $\mathfrak C^H(\mathcal M)\cap  \mathfrak C^H({_B}HOM(A,\mathcal N'))=0$ in $ \mathfrak C^H({_B}HOM(A,\mathcal N))$. By Proposition \ref{P7.11sgs}, we have 
$\mathfrak C^H({_B}HOM(A,\mathcal  N'))\cong \mathcal N'$ and $\mathfrak C^H({_B}HOM(A,\mathcal N))\cong \mathcal N$. Since $\mathcal N'\hookrightarrow\mathcal N$ is essential, we must now have $\mathfrak C^H(\mathcal M)=0$. From part (a), we now get $\mathcal M=0$. 
\end{proof}

We know that both ${_A}\mathfrak S^H$ and ${_B}\mathfrak S$ are Grothendieck categories. For an object $\mathcal M\in {_A}\mathfrak S^H$, we denote by
${_A}\mathcal E^H(\mathcal M)$ the injective envelope of $\mathcal M$. Similarly, for $\mathcal N\in {_B}\mathfrak S$, we denote by 
${_B}\mathcal E(\mathcal N)$ the injective envelope of $\mathcal N$. We have now reached the main result of this section.

\begin{Thm}\label{T713}
Let $\mathcal N\in {_B}\mathfrak S$. Then, we have
\begin{equation}\label{iso711}
{_A}\mathcal E^H({_B}HOM(A,\mathcal N))\cong {_B}HOM(A,{_B}\mathcal E(\mathcal N))\qquad \mathfrak C^H({_A}\mathcal E^H({_B}HOM(A,\mathcal N)))\cong {_B}\mathcal E(\mathcal N)
\end{equation}
\end{Thm}

\begin{proof}
By Corollary \ref{C7.9}, we know that the functor ${_B}HOM(A,\_\_):{_B}\mathfrak S\longrightarrow {_A}\mathfrak S^H$ preserves  injectives and by Lemma \ref{L7.12aq} we know that it preserves essential monomorphisms. This gives ${_A}\mathcal E^H({_B}HOM(A,\mathcal N))\cong {_B}HOM(A,{_B}\mathcal E(\mathcal N))$. The second isomorphism in \eqref{iso711} now follows by applying Proposition \ref{P7.11sgs}. 
\end{proof}

\section{Elementary objects, injective envelopes and the coinduction functor}

We continute with $H$ being a cosemisimple Hopf algebra, $A$ a right $H$-comodule algebra and $B$ being the subalgebra of $H$-coinvariants of $A$.  By Proposition \ref{P7.8vk}, we have an adjunction 
\begin{equation}
{_B}\mathfrak S(\mathfrak C^H(\mathcal M),\mathcal N)\cong {_A}\mathfrak S^H(\mathcal M, {_B}HOM(A,\mathcal N))
\end{equation} for $\mathcal M\in {_A}\mathfrak S^H$ and $\mathcal N\in {_B}\mathfrak S$.  In Theorem \ref{T713}, we showed that $\mathfrak C^H ({_A}\mathcal E^H({_B}Hom(A,\mathcal N)))\cong {_B}\mathcal E(\mathcal N)$ for $\mathcal N\in {_B}\mathfrak S$.  The first aim of this section is to relate the injective envelope ${_A}\mathcal E^H(\mathcal M)$ of $\mathcal M\in {_A}\mathfrak S^H$ to ${_B}HOM(A,{_B}\mathcal E(\mathfrak C^H(\mathcal M))$.  As in \eqref{7.2eq}, we consider the morphism
\begin{equation}\label{7.2eqx}
\pi_{\mathcal P}^\phi: \mathcal P\xrightarrow{\Delta_{\mathcal P}^H}\mathcal P\otimes H \xrightarrow{\mathcal P\otimes \phi}\mathcal P
\end{equation} for any $(\mathcal P,\Delta_{\mathcal P}^H)\in \mathfrak S^H$ and recall from the proof of Proposition \ref{P7.2} that $\mathfrak K^H(\mathcal P)=Ker(\pi_{\mathcal P}^\phi)$. In particular, if $\mathcal M\in  {_A}\mathfrak S^H$, we note that $\mathfrak K^H(\mathcal M)=Ker(\pi_{\mathcal M}^\phi)
\in \mathfrak S^H$. For  $\mathcal M\in  {_A}\mathfrak S^H$, we now let ${_A}\mathfrak K^H(\mathcal M)$ denote the sum of all subobjects 
$\mathcal M'\subseteq \mathcal M$ in $  {_A}\mathfrak S^H$ such that  $\mathcal M'\subseteq \mathfrak K^H(\mathcal M)$. We now consider the following pair of full subcategories of ${_A}\mathfrak S^H$:
\begin{equation}\label{tor8}
 \mathcal T({_A}\mathfrak S^H):=\{\mbox{$\mathcal M\in {_A}\mathfrak S^H$ $\vert$ $\mathfrak C^H(\mathcal M)=0$ }\} \qquad \mathcal F({_A}\mathfrak S^H):=\{\mbox{$\mathcal M\in {_A}\mathfrak S^H$ $\vert$ ${_A}\mathfrak K^H(\mathcal M)=0$ }\}
\end{equation} In other words,  $\mathcal M\in   {_A}\mathfrak S^H$ lies in  the full subcategory $ \mathcal F({_A}\mathfrak S^H)$ if for any
$\mathcal M'\subseteq \mathcal M\in  {_A}\mathfrak S^H$ such that $\mathcal M'\subseteq Ker(\pi_{\mathcal M}^\phi)$ in $\mathfrak S^H$, we must have
$\mathcal M'=0$.

\begin{lem}\label{L8.01p}
$ \mathcal T({_A}\mathfrak S^H)$ is closed under quotients and $ \mathcal F({_A}\mathfrak S^H)$ is closed under subobjects. If $\mathcal M
\in  \mathcal T({_A}\mathfrak S^H)$ and $\mathcal N\in  \mathcal F({_A}\mathfrak S^H)$, we have ${_A}\mathfrak S^H(\mathcal M,\mathcal N)=0$.
\end{lem}

\begin{proof}
Let $\mathcal M\twoheadrightarrow\mathcal M''$ be an epimorphism in ${_A}\mathfrak S^H$ with $\mathcal M\in \mathcal T({_A}\mathfrak S^H)$. Then, $\mathcal M\longrightarrow \mathcal M''$ is also an epimorphism in $\mathfrak S^H$. As in the proof of Lemma \ref{quo7}, it follows that $\mathcal C^H(\mathcal 
M'')=0$, which shows that $\mathcal M''\in \mathcal T({_A}\mathfrak S^H)$. On the other hand, let $\mathcal N'\hookrightarrow\mathcal N$ be a monomorphism in ${_A}\mathfrak S^H$ with $\mathcal N\in \mathcal F({_A}\mathfrak S^H)$.  If we take $\mathcal N''\subseteq \mathcal N'$ in ${_A}\mathfrak S^H$ such that
$\mathcal N''\subseteq Ker(\pi^\phi_{\mathcal N'})$, then we have $\mathcal N''\subseteq Ker(\pi^\phi_{\mathcal N'})\subseteq Ker(\pi^\phi_{\mathcal N})$. Since $\mathcal N\in \mathcal F({_A}\mathfrak S^H)$, we now get $\mathcal N''=0$ and hence $\mathcal N''\in \mathcal F({_A}\mathfrak S^H)$.

\smallskip
We now claim that $\mathcal T({_A}\mathfrak S^H)\cap \mathcal F({_A}\mathfrak S^H)=0$. Indeed, if $\mathcal P\in \mathcal T({_A}\mathfrak S^H)
\cap \mathcal F({_A}\mathfrak S^H)$, we have $\mathfrak C^H(\mathcal P)=0$. Then, $\mathfrak K^H(\mathcal P)=Ker(\pi_{\mathcal P}^\phi)=\mathcal P$. Since
$\mathcal P\in  \mathcal F({_A}\mathfrak S^H)$, it now follows that $\mathcal P=0$. Hence, if   we have $\mathcal M
\in  \mathcal T({_A}\mathfrak S^H)$, $\mathcal N\in  \mathcal F({_A}\mathfrak S^H)$, and $\psi\in {_A}\mathfrak S^H(\mathcal M,\mathcal N)$, it follows from the above that
$Im(\psi)\in  \mathcal T({_A}\mathfrak S^H)
\cap \mathcal F({_A}\mathfrak S^H)$, i.e., $Im(\psi)=0$.
\end{proof} 

The next result gives us another description of ${_A}\mathfrak K^H(\mathcal U)$.

\begin{lem}\label{L8.02p} Let $\mathcal M\in {_A}\mathfrak S^H$. Then, ${_A}\mathfrak K^H(\mathcal M)$ is the kernel of the canonical morphism
$\mathcal M\longrightarrow {_B}HOM(A,\mathfrak C^H(\mathcal M))$. 
\end{lem}

\begin{proof}
By definition, the kernel of the canonical morphism $\mathcal M\longrightarrow {_B}HOM(A,\mathfrak C^H(\mathcal M))$ is the sum of all subobjects 
$\mathcal M'\subseteq \mathcal M$ in ${_A}\mathfrak S^H$ such that the composition $\mathcal M'\longrightarrow \mathcal M\longrightarrow {_B}HOM(A,\mathfrak C^H(\mathcal M))$ is zero. Using the adjunction in Proposition \ref{P7.8vk}, this is equivalent to the composition $\mathfrak C^H(\mathcal M')\longrightarrow \mathfrak C^H(\mathcal M)
\xrightarrow{id}\mathfrak C^H(\mathcal M)$ being zero, i.e., $\mathfrak C^H(\mathcal M')=0$. But $\mathfrak C^H(\mathcal M')=0$ for a subobject $\mathcal M'\subseteq \mathcal M$ in ${_A}\mathfrak S^H$ if and only if $\mathcal M'\subseteq \mathfrak K^H(\mathcal M)=Ker(\pi^\phi_{\mathcal M})$. The result is now clear from the definition
of ${_A}\mathfrak K^H(\mathcal M)$.
\end{proof}

\begin{thm}\label{P8.1dc}
Let $A$ be a right $H$-comodule algebra. Then, the pair $( \mathcal T({_A}\mathfrak S^H), \mathcal F({_A}\mathfrak S^H))$ gives a torsion theory on ${_A}\mathfrak S^H$.
\end{thm}

\begin{proof}
We consider $\mathcal M\in {_A}\mathfrak S^H$. We claim that ${_A}\mathfrak K^H(\mathcal M)\in  \mathcal T({_A}\mathfrak S^H)$. Since ${_A}\mathfrak K^H(\mathcal M)\subseteq Ker(\pi^\phi_{\mathcal M})$, the restriction of $\pi^\phi_{\mathcal M}$ to ${_A}\mathfrak K^H(\mathcal M)\subseteq \mathcal M$ is $0$, i.e.,
$\pi^\phi_{{_A}\mathfrak K^H(\mathcal M)}=0$. We know that $\mathfrak C^H({_A}\mathfrak K^H(\mathcal M))=Im(\pi^\phi_{{_A}\mathfrak K^H(\mathcal M)})$, which shows that
$ \mathfrak C^H({_A}\mathfrak K^H(\mathcal M))=0$, i.e., ${_A}\mathfrak K^H(\mathcal M)\in  \mathcal T({_A}\mathfrak S^H)$. 

\smallskip
Also since $\mathfrak C^H({_A}\mathfrak K^H(\mathcal M))=0$, and $\mathfrak C^H$ is exact, we see that $\mathfrak C^H(\mathcal M)=\mathfrak C^H(\mathcal M/{_A}\mathfrak K^H(\mathcal M))$. From Lemma \ref{L8.02p}, we now have a monomorphism
\begin{equation}\label{8.4rq}
\mathcal M/{_A}\mathfrak K^H(\mathcal M)\hookrightarrow {_B}HOM(A,\mathfrak C^H(\mathcal M))= {_B}HOM(A,\mathfrak C^H(\mathcal M/{_A}\mathfrak K^H(\mathcal M))
\end{equation} Applying Lemma \ref{L8.02p} again, it follows that ${_A}\mathfrak K^H(\mathcal M/{_A}\mathfrak K^H(\mathcal M))=0$, which shows that $\mathcal M/{_A}\mathfrak K^H(\mathcal M)\in \mathcal F({_A}\mathfrak S^H)$. Additionally, we know from Lemma \ref{L8.01p} that  ${_A}\mathfrak S^H(\mathcal M,\mathcal N)=0$ for any $\mathcal M\in  \mathcal T({_A}\mathfrak S^H)$ and $\mathcal N\in  \mathcal F({_A}\mathfrak S^H)$. This proves the result.
\end{proof}

\begin{lem}\label{L8.04p} Let $\mathcal M\in  {_A}\mathfrak S^H$ with ${_A} \mathfrak K^H(\mathcal M)=0$.  Then, the canonical morphism $\mathcal M\longrightarrow {_B}HOM(A, \mathfrak C^H(\mathcal M))$ is an essential monomorphism in ${_A}\mathfrak S^H$.

\end{lem}

\begin{proof}
Since ${_A}\mathfrak K^H(\mathcal M)=0$, it follows by Lemma \ref{L8.02p} that  the canonical morphism $\mathcal M\longrightarrow {_B}HOM(A,\mathfrak C^H(\mathcal M))$ is a  monomorphism ${_A}\mathfrak S^H$. We now consider some $0\ne \mathcal N\subseteq {_B}HOM(A,\mathfrak C^H(\mathcal M))$ in ${_A}\mathfrak S^H$. By Lemma 
\ref{L7.12aq}(a), we have $\mathfrak C^H(\mathcal N)\ne 0$. Since $\mathfrak C^H$ is exact, we now have
$ \mathfrak C^H(\mathcal N)\subseteq \mathfrak C^H( {_B}HOM(A,\mathfrak C^H(\mathcal M)))=\mathfrak C^H(\mathcal M)$ where the latter equality follows from Proposition \ref{P7.11sgs}. Since $\mathfrak C^H$ is exact, we also see that
\begin{equation}
\mathfrak C^H(\mathcal N\cap \mathcal M)=\mathfrak C^H(\mathcal N)\cap \mathfrak C^H(\mathcal M)=\mathfrak C^H(\mathcal N)\ne 0
\end{equation} which shows that $\mathcal N\cap \mathcal M\ne 0$. 
\end{proof}

\begin{Thm}\label{T8.5gv}  Let $\mathcal M\in  {_A}\mathfrak S^H$ be torsion free, i.e., $\mathcal M\in \mathcal F({_A}\mathfrak S^H)$.  Then,  we have
${_A}\mathcal E^H(\mathcal M)\cong {_B}HOM(A,{_B}\mathcal E(\mathfrak C^H(\mathcal M)))$.

\end{Thm}

\begin{proof}
Since ${_A}\mathfrak K^H(\mathcal M)=0$, it follows by Lemma \ref{L8.04p} that  the canonical morphism $\mathcal M\longrightarrow {_B}HOM(A,\mathfrak C^H(\mathcal M))$ is an essential  monomorphism ${_A}\mathfrak S^H$. We know that $\mathfrak C^H(\mathcal M)\hookrightarrow {_B}\mathcal E(\mathfrak C^H(\mathcal M))$ is an essential monomorphism, and from Lemma \ref{L7.12aq}(b) that ${_B}HOM(A,\_\_):{_B}\mathfrak S\longrightarrow {_A}\mathfrak S^H$ preserves essential monomorphisms. Accordingly, we have an essential monomorphism
\begin{equation}
\mathcal M\longrightarrow {_B}HOM(A,\mathfrak C^H(\mathcal M))\longrightarrow {_B}HOM(A,{_B}\mathcal E(\mathfrak C^H(\mathcal M)))
\end{equation} By Corollary \ref{C7.9}, the functor ${_B}HOM(A,\_\_):{_B}\mathfrak S\longrightarrow {_A}\mathfrak S^H$ preserves injectives. The result is now clear.
\end{proof}

The condition in Theorem \ref{T8.5gv} leads us to look at conditions under which all objects in ${_A}\mathfrak S^H$ are torsion free, i.e., $\mathcal F({_A}\mathfrak S^H)={_A}\mathfrak S^H$. We recall that an object $E$ in an abelian category $\mathfrak A$ is said to be a cogenerator if it satisfies the following condition: given any epimorphism $X\twoheadrightarrow Y$ in 
$\mathfrak A$ that is not an isomorphism, there is a morphism $X\longrightarrow E$ that does not factor through $Y$. 
We now need the following general result.

\begin{lem}\label{L8.55tc}
Let $L:\mathfrak A\longrightarrow \mathfrak B$ be a functor between abelian categories, which has a right adjoint $R:\mathfrak B\longrightarrow \mathfrak A$. Suppose that $L$ has the property that if $p:X\twoheadrightarrow Y$ is an epimorphism in $\mathfrak A$ which is not an isomorphism, then $L(p):L(X)\longrightarrow L(Y)$ is not an isomorphism.
Then, if $E$ is a cogenerator in $\mathfrak B$, its image $R(E)$ is a cogenerator in $\mathfrak A$.
\end{lem}

\begin{proof} Suppose that $R(E)$ is not a cogenerator in $\mathfrak A$. Then, there exists an 
 epimorphism  $p:X\twoheadrightarrow Y$ in $\mathfrak A$ which is not an isomorphism, but any morphism $X\longrightarrow R(E)$ must factor through
$Y$. 

\smallskip Since $L$ is a left adjoint, it preserves epimorphisms.  From the assumption on $L$, it now follows that $L(p):L(X)\longrightarrow L(Y)$ is an epimorphism in $\mathfrak B$ that is not an isomorphism. Since $E$ is a cogenerator in $\mathfrak B$, we can now find a morphism
$f:L(X)\longrightarrow E$ which does not factor through $L(Y)$.  By the adjunction, we now have $g:X\longrightarrow R(E)$ corresponding to $f$.  Then, there exists $h:Y\longrightarrow R(E)$ such that $g=h\circ p$. But then we must have $f=h'\circ L(p)$, where $h':L(Y)\longrightarrow E$ corresponds to $h:Y\longrightarrow R(E)$ using the adjunction.  This contradicts the assumption that $f$ does not factor through $L(Y)$.
\end{proof}

For the remainder of this paper,  we suppose that $\mathfrak S$ has an injective cogenerator $\mathcal I$.  
We also note that for any $\mathcal M\in \mathfrak S$, the multiplication $\mu_A:A\otimes A\longrightarrow A$ induces a canonical morphism in $\mathfrak S$
\begin{equation}
\underline{Hom}(A,\mathcal M)\longrightarrow \underline{Hom}(A\otimes A,\mathcal M)= \underline{Hom}(A,\underline{Hom}(A,\mathcal M))
\end{equation} which corresponds by adjointness to $A\otimes \underline{Hom}(A,\mathcal M)\longrightarrow \underline{Hom}(A,\mathcal M)$. It may now be verified that $\underline{Hom}(A,\mathcal M)\in {_A}\mathfrak S$. 

\begin{lem}\label{L8.6vq}
Suppose that $\mathfrak S$ has an injective cogenerator $\mathcal I$.  Then, $\underline{Hom}(A,\mathcal I)$ is an injective cogenerator for the category ${_A}\mathfrak S$.
\end{lem}

\begin{proof} We may verify directly that for any $\mathcal N\in {_A}\mathfrak S$ and $\mathcal M\in \mathfrak S$, we have
\begin{equation}\label{e8.8uh}
 \mathfrak S(\mathcal N,\mathcal M)\cong {_A}\mathfrak S(\mathcal N,\underline{Hom}(A,\mathcal M))
\end{equation} Since $\mathcal I\in \mathfrak S$ is injective, $\mathfrak S(\_\_,\mathcal I)$ is exact. From \eqref{e8.8uh}, it is now clear that ${_A}\mathfrak S(\_\_,\underline{Hom}(A,\mathcal I))$ is exact, i.e., $\underline{Hom}(A,\mathcal I)$ is an injective in ${_A}\mathfrak S$.  It is evident that the forgetful functor 
${_A}\mathfrak S\longrightarrow \mathfrak S$ satisfies the condition in Lemma \ref{L8.55tc}. Accordingly, since $\mathcal I$ is a cogenerator in
$\mathfrak S$, the object $\underline{Hom}(A,\mathcal I)$ is a cogenerator in ${_A}\mathfrak S$.

\end{proof}

\begin{lem}\label{L8.65rt} There are natural isomorphisms
\begin{equation}\label{865rt}
 {_A}\mathfrak S(\mathcal N,\mathcal M)\cong  {_A}\mathfrak S^H(\mathcal N,\mathcal M\otimes H)
\end{equation}
for $(\mathcal N,\mu_{\mathcal N}^A,\Delta_{\mathcal N}^H)\in {_A}\mathfrak S^H$ and $(\mathcal M,\mu_{\mathcal M}^A)\in {_A}\mathfrak S$, giving an adjunction
of functors.
\end{lem}

\begin{proof}
We consider $\phi\in {_A}\mathfrak S(\mathcal N,\mathcal M)$, which induces $\phi':=(\phi\otimes H)\circ \Delta_{\mathcal N}^H\in  {_A}\mathfrak S^H(\mathcal N,\mathcal M\otimes H)$. On the other hand, given $\psi\in  {_A}\mathfrak S^H(\mathcal N,\mathcal M\otimes H)$, we have
$\psi':=(\mathcal M\otimes \epsilon_H)\circ \psi\in {_A}\mathfrak S(\mathcal N,\mathcal M)$.  We now verify that
\begin{equation}
(\mathcal M\otimes \epsilon_H)\circ (\phi\otimes H)\circ \Delta_{\mathcal N}^H=\phi\circ (\mathcal N\otimes \epsilon_H)\circ \Delta_{\mathcal N}^H=\phi
\end{equation} We also have
\begin{equation}
(\mathcal M\otimes \epsilon_H\otimes H)\circ (\psi\otimes H)\circ \Delta_{\mathcal N}^H=(\mathcal M\otimes \epsilon_H\otimes H)\circ (\mathcal M
\otimes \Delta_H)\circ \psi=\psi
\end{equation}
This proves the result.
\end{proof}

\begin{thm}\label{P8.7yb}
Suppose that $\mathfrak S$ has an injective cogenerator $\mathcal I$.  Then, $\underline{Hom}(A,\mathcal I)\otimes H$ is an injective cogenerator for the category ${_A}\mathfrak S^H$.
\end{thm}

\begin{proof}
We know from Lemma 
\ref{L8.6vq} that $\underline{Hom}(A,\mathcal I)$ is an injective cogenerator for the category ${_A}\mathfrak S$. Since the forgetful functor ${_A}\mathfrak S^H\longrightarrow {_A}\mathfrak S$ is exact, its right adjoint preserves injectives. By the adjunction in Lemma \ref{L8.65rt}, it now follows that $\underline{Hom}(A,\mathcal I)\otimes H$ is an injective object in ${_A}\mathfrak S^H$. It is also clear that the forgetful functor 
${_A}\mathfrak S^H\longrightarrow {_A}\mathfrak S$ satisfies the condition in Lemma \ref{L8.55tc}. Accordingly, since $\underline{Hom}(A,\mathcal I)$ is a cogenerator in
${_A}\mathfrak S$, the object $\underline{Hom}(A,\mathcal I)\otimes H$ is a cogenerator in ${_A}\mathfrak S^H$.
\end{proof}

\begin{Thm}\label{alpha}
Suppose that $\underline{Hom}(A,\mathcal I)\otimes H\in \mathcal F({_A}\mathfrak S^H)$. Then, $\mathcal F({_A}\mathfrak S^H)={_A}\mathfrak S^H$, i.e., every object $\mathcal M\in {_A}\mathfrak S^H$ satisfies ${_A}\mathfrak K^H(\mathcal M)=0$.
\end{Thm}
\begin{proof}
By Proposition \ref{P8.7yb}, we know that $\underline{Hom}(A,\mathcal I)\otimes H$ is an injective cogenerator for the category ${_A}\mathfrak S^H$. In other words, every object $\mathcal M\in {_A}\mathfrak S^H$ is a subobject of a direct product of copies of $\underline{Hom}(A,\mathcal I)\otimes H$. By Proposition \ref{P8.1dc}, we know that $( \mathcal T({_A}\mathfrak S^H),  \mathcal F({_A}\mathfrak S^H))$ is a torsion theory on ${_A}\mathfrak S^H$. Hence, $\mathcal F({_A}\mathfrak S^H)$ is closed under products and subobjects. Since $\underline{Hom}(A,\mathcal I)\otimes H\in \mathcal F({_A}\mathfrak S^H)$, the result follows.
\end{proof}

\begin{cor}\label{C8.75s}
Suppose that $\underline{Hom}(A,\mathcal I)\otimes H\in \mathcal F({_A}\mathfrak S^H)$. Then, if $\mathcal E\in {_A}\mathfrak S^H$ is an injective object, 
$\mathfrak C^H(\mathcal E)$ is an injective object in ${_B}\mathfrak S$.
\end{cor}
\begin{proof}
We know that the inclusion $\mathfrak C^H(\mathcal E)\hookrightarrow {_B}\mathcal E(\mathfrak C^H(\mathcal E))$ into the injective envelope of $\mathfrak C^H(\mathcal E)$ is an essential monomorphism in ${_B}\mathfrak S$. Applying Lemma \ref{L7.12aq}(b), ${_B}HOM(A,\mathfrak C^H(\mathcal E))\hookrightarrow {_B}HOM(A,
{_B}\mathcal E(\mathfrak C^H(\mathcal E)))$ is an essential monomorphism in ${_A}\mathcal S^H$.  Since $\mathcal F({_A}\mathfrak S^H)={_A}\mathfrak S^H$, we know that  ${_A}\mathfrak K^H(\mathcal E)=0$. Combining with Lemma \ref{L8.04p}, we now have the essential monomorphism
\begin{equation}\label{8.12eqk}
\mathcal E\hookrightarrow {_B}HOM(A,\mathfrak C^H(\mathcal E))\hookrightarrow {_B}HOM(A,
{_B}\mathcal E(\mathfrak C^H(\mathcal E)))
\end{equation} But since $\mathcal E\in {_A}\mathfrak S^H$ is injective, it follows that the essential monomorphisms in \eqref{8.12eqk} are all isomorphisms. Using Proposition \ref{P7.11sgs}, we now have
\begin{equation}
\mathfrak C^H(\mathcal E)\cong \mathfrak C^H({_B}HOM(A,
{_B}\mathcal E(\mathfrak C^H(\mathcal E))))\cong {_B}\mathcal E(\mathfrak C^H(\mathcal E))
\end{equation}  This proves the result.
\end{proof}

We will also use Theorem \ref{alpha} and Corollary \ref{C8.75s} to promote the isomorphism in Proposition \ref{P7.10gs} to the level of derived functors. 

\begin{thm}\label{P8.76d}
Let $A$ be a right $H$-comodule algebra and $B=A^{co H}$. Suppose that $\underline{Hom}(A,\mathcal I)\otimes H\in \mathcal F({_A}\mathfrak S^H)$.  Then, we have natural isomorphisms
\begin{equation}\label{7.7ytr}
Ext_{{_A}\mathfrak S^H}^{\bullet}(A\otimes_B\mathcal N,\mathcal M)\cong  Ext_{{_B}\mathfrak S}^\bullet(\mathcal N,\mathfrak C^H(\mathcal M)) 
\end{equation}
for $\mathcal N\in {_B}\mathfrak S$ and $\mathcal M\in {_A}\mathfrak S^H$.
\end{thm}

\begin{proof}
Let $\mathcal M\longrightarrow \mathcal E^\bullet$ be an injective resolution in ${_A}\mathfrak S^H$. Accordingly, the derived functors $Ext_{{_A}\mathfrak S^H}^{\bullet}(A\otimes_B\mathcal N,\mathcal M)$ may be computed by taking cohomologies of the complex ${_A}\mathfrak S^H(A\otimes_B\mathcal N,\mathcal E^\bullet)\cong {_B}\mathfrak S(\mathcal N,\mathfrak C^H(\mathcal E^\bullet))$. Since $\underline{Hom}(A,\mathcal I)\otimes H\in \mathcal F({_A}\mathfrak S^H)$,  we know from Corollary \ref{C8.75s} that $\mathfrak C^H$ preserves injectives. Since $\mathfrak C^H$ is exact, it now follows that $\mathfrak C^H(\mathcal E^\bullet)$ is an injective resolution of $\mathfrak C^H(\mathcal M)$. The result is now clear. 
\end{proof}

\begin{lem}\label{L8.8fj}
Any injective object in ${_A}\mathfrak S^H$ can be expressed as a direct summand of an injective of the form $\mathcal E'\otimes H$, where $\mathcal E'\in {_A}\mathfrak S$ is an injective object.
\end{lem}

\begin{proof} Let $(\mathcal E,\mu_{\mathcal E}^A,\Delta_{\mathcal E}^H)\in {_A}\mathfrak S^H$ be an injective object. We note that the morphism $\Delta_{\mathcal E}^H:\mathcal E\longrightarrow \mathcal E\otimes H$ is a morphism in ${_A}\mathfrak S^H$. We also note that $(\mathcal E\otimes \epsilon_H)\circ \Delta^H_{\mathcal E}=id$ in ${_A}\mathfrak S$, which makes $\Delta^H_{\mathcal E}$ a monomorphism in ${_A}\mathfrak S$. Since kernels of morphisms in ${_A}\mathfrak S^H$ are computed in $\mathfrak S$, it follows that $ \Delta_{\mathcal E}^H:\mathcal E\longrightarrow \mathcal E\otimes H$ is a monomorphism in ${_A}\mathfrak S^H$.  

\smallskip
Since ${_A}\mathfrak S$ is a Grothendieck category, we may  choose an injective $\mathcal E'\in {_A}\mathfrak S$ along with a monomorphism
$\mathcal E\hookrightarrow \mathcal E'$ in ${_A}\mathfrak S$. This induces a monomorphism $\mathcal E\otimes H\hookrightarrow \mathcal E'\otimes H$ in ${_A}\mathfrak S^H$. Since $\mathcal E\in {_A}\mathfrak S^H$ is injective, the inclusion $\mathcal E\overset{ \Delta_{\mathcal E}^H}{\hookrightarrow}\mathcal E\otimes H\hookrightarrow \mathcal E'\otimes H$ in ${_A}\mathfrak S^H$ makes $\mathcal E$ a direct summand of $ \mathcal E'\otimes H$. From the exactness of the forgetful functor 
${_A}\mathfrak S^H\longrightarrow {_A}\mathfrak S$ and the adjunction in Lemma \ref{L8.65rt}, we know that its right adjoint preserves injectives. Since $\mathcal E'\in {_A}\mathfrak S$ is injective, it now follows that $\mathcal E'\otimes H\in {_A}\mathfrak S^H$ is injective. This proves the result.
\end{proof}

By adapting the terminology of \cite{AZ}, we will now say that the category $\mathfrak S$ is strongly locally left notherian  if ${_R}\mathfrak S$ is locally noetherian for any left noetherian $k$-algebra $R$. 

\begin{thm}
\label{P8.9wt} Suppose that $\mathfrak S$ is strongly locally left noetherian. Let $A$ be a left noetherian $k$-algebra that is also a right $H$-comodule algebra. Then, the direct sum of injectives in
${_A}\mathfrak S^H$ is injective.
\end{thm}
\begin{proof} Let $\{\mathcal E_i\}_{i\in I}$ be a family of injectives in ${_A}\mathfrak S^H$. Applying Lemma \ref{L8.8fj}, we choose for each $i\in I$ an injective $ \mathcal E_i'\in {_A}\mathfrak S$ such that $\mathcal E_i$ is a direct summand of $\mathcal E_i'\otimes H$ in ${_A}\mathfrak S^H$. Then, $\bigoplus_{i\in I}
\mathcal E_i$ is a direct summand of $\bigoplus_{i\in I}\mathcal E_i'\otimes H$ in ${_A}\mathfrak S^H$.

\smallskip
From the assumptions, it follows that ${_A}\mathfrak S$ is a locally noetherian category. Hence, the direct sum  $\bigoplus_{i\in I}\mathcal E_i'$ of injective objects is injective in ${_A}\mathfrak S$. Again, the right adjoint in Lemma \ref{L8.65rt} now makes $(\bigoplus_{i\in I}\mathcal E_i')\otimes H$ an injective in ${_A}\mathfrak S^H$. Accordingly, the direct summand $\bigoplus_{i\in I}
\mathcal E_i$ of $(\bigoplus_{i\in I}\mathcal E_i')\otimes H$ is injective in ${_A}\mathfrak S^H$.
\end{proof}

We now set $\hat{\mathcal I}:=\underline{Hom}(A,\mathcal I)\otimes H\in {_A}\mathfrak S^H$. For the remainder of this section, we will  make the following assumptions:

\smallskip
(1) $A$ is a  left noetherian $k$-algebra and a right $H$-comodule algebra

\smallskip
(2) $\mathfrak S$ is strongly locally left noetherian.

\smallskip
(3) $\mathfrak S$ has an injective cogenerator $\mathcal I$ which satisfies 
 $\hat{\mathcal I}=\underline{Hom}(A,\mathcal I)\otimes H\in  \mathcal F({_A}\mathfrak S^H)$.  By Theorem \ref{alpha}, it  follows therefore that every object $\mathcal M\in {_A}\mathfrak S^H$ satisfies ${_A}\mathfrak K^H(\mathcal M)=0$.

\begin{lem}\label{L8.10db}
Let $\{\mathcal E_i\}_{i\in I}$ be a family of injectives in ${_B}\mathfrak S$. Then, we have an isomorphism
\begin{equation}\label{bsiso}
\underset{i\in I}{\bigoplus}{_B}HOM(A,\mathcal E_i)\cong {_B}HOM\left(A,\underset{i\in I}{\bigoplus}\mathcal E_i\right)
\end{equation}
\end{lem}

\begin{proof} We set $\mathcal E:=\underset{i\in I}{\bigoplus}{_B}HOM(A,\mathcal E_i)$. Since ${_B}HOM(A,\_\_)$ preserves injectives, each $ {_B}HOM(A,\mathcal E_i)\in {_A}\mathfrak S^H$ is injective. Using Proposition \ref{P8.9wt}, the direct sum $\mathcal E=\underset{i\in I}{\bigoplus}{_B}HOM(A,\mathcal E_i)$ is injective in ${_A}\mathfrak S^H$. Since
$ {_A}\mathfrak S^H=\mathcal F({_A}\mathfrak S^H)$, we know that ${_A}\mathfrak K^H(\mathcal E)=0$. By Lemma \ref{L8.04p}, the canonical morphism $\mathcal E\longrightarrow {_B}HOM(A,\mathfrak C^H(\mathcal E))$ is an essential monomorphism in ${_A}\mathfrak S^H$. Since $\mathcal E\in {_A}\mathfrak S^H$ is an injective object, it now follows that $\mathcal E\cong  {_B}HOM(A,\mathfrak C^H(\mathcal E))$. Using Proposition \ref{P7.11sgs} and the fact that $\mathfrak C^H$ preserves direct sums, we now have
\begin{equation}
\mathfrak C^H(\mathcal E)=\underset{i\in I}{\bigoplus}\textrm{ }\mathfrak C^H\left({_B}HOM(A,\mathcal E_i)\right)\cong \underset{i\in I}{\bigoplus} \mathcal E_i
\end{equation} The result now follows by applying ${_B}HOM(A,\_\_)$.

\end{proof}

\begin{lem}\label{L8.16cop}  Let $\mathcal M\longrightarrow {_B}HOM(A,\mathcal N)$ be an essential monomorphism with $\mathcal M\in {_A}\mathfrak S^H$ and $\mathcal 
N\in {_B}\mathfrak S$. Then, $\mathfrak C^H(\mathcal M)\longrightarrow \mathfrak C^H({_B}HOM(A,\mathcal N))=\mathcal N$ is an essential monomorphism in ${_B}\mathfrak S$.

\end{lem}

\begin{proof} Using Proposition \ref{P7.11sgs} and the fact that $\mathfrak C^H$ is exact, we have an inclusion $\mathfrak C^H(\mathcal M)\hookrightarrow \mathfrak C^H({_B}HOM(A,\mathcal N))=\mathcal N$. We suppose there exists $0\ne \mathcal M'\subseteq \mathcal N$ such that $\mathcal M'\cap \mathfrak C^H(\mathcal M)=0$. Since ${_B}HOM(A,\_\_)$ is a right adjoint, this gives us
\begin{equation}\label{817qew}
0={_B}HOM(A,\mathfrak C^H(\mathcal M))\cap {_B}HOM(A,\mathcal M')\subseteq {_B}HOM(A,\mathcal N)
\end{equation} Since $ {_A}\mathfrak K^H(\mathcal M)=0$, we know that $\mathcal M\subseteq {_B}HOM(A,\mathfrak C^H(\mathcal M))$ and it follows from \eqref{817qew} that $\mathcal M
\cap  {_B}HOM(A,\mathcal M')=0$. We are given that $\mathcal M\hookrightarrow {_B}HOM(A,\mathcal N)$ is an essential monomorphism, which now shows that 
${_B}HOM(A,\mathcal M')=0$. Applying  Proposition \ref{P7.11sgs} again, we have $\mathcal M'=\mathfrak C^H({_B}HOM(A,\mathcal M'))=0$.

\end{proof}

For the final result of this paper, we will need to recall some concepts from our work in \cite{BanKr}. We note that if $T$ is a commutative noetherian $k$-algebra and $\mathcal M$ is an object of ${_T}\mathfrak S$, we may consider the annihilator
\begin{equation}
Ann_{{_T}\mathfrak S}(\mathcal M):=\{\mbox{$a\in R$ $\vert$ $a_{\mathcal M}=0:\mathcal M\xrightarrow{\quad a\cdot \quad}\mathcal M$}\}
\end{equation} It is clear that $Ann_{{_T}\mathfrak S}(\mathcal M)$ is an ideal in $T$. 

\begin{defn}\label{D8.16} (see \cite[Definition 2.1]{BanKr})
Let $\mathfrak S$ be a strongly locally noetherian $k$-linear category and let $T$ be a commutative noetherian $k$-algebra. We will say that an object $0\ne \mathcal L\in {_T}\mathfrak S$ is $T$-elementary if it satisfies the following conditions:

\smallskip
(a) $\mathcal L$ is finitely generated as an object of ${_T}\mathfrak S$, i.e., the functor ${_T}\mathfrak S(\mathcal L,\_\_)$ preserves filtered colimits of monomorphisms.

\smallskip
(b) There exists a prime ideal $\mathfrak p\subseteq T$ such that any non-zero subobject $\mathcal  L'\subseteq \mathcal L$ satisfies $Ann_{{_T}\mathfrak S}(\mathcal L')=\mathfrak p$. 
\end{defn}

Suppose that $\mathcal M\in {_T}\mathfrak S$ and we have a $T$-elementary object $\mathcal L\subseteq \mathcal M$ such that $Ann_{{_T}\mathfrak S}(\mathcal L)=\mathfrak p$. Then, we will say that $\mathfrak p$ is an associated prime of $\mathcal M$. The set of associated primes of $\mathcal M$ is denoted by $Ass_{{_T}\mathfrak S}(\mathcal M)$. In this paper, we will need a few basic properties of the theory of  associated primes for objects in ${_T}\mathfrak S$ that we have developed in \cite{BanKr}.

\smallskip
(1) For any $0\ne \mathcal M\in {_T}\mathfrak S$, we have $Ass_{{_T}\mathfrak S}(\mathcal M)\ne \phi$.

\smallskip
(2) If $0\longrightarrow \mathcal M'\longrightarrow \mathcal M\longrightarrow \mathcal M''\longrightarrow 0$ is a short exact sequence in ${_T}\mathfrak S$, we have 
\begin{equation*} Ass_{{_T}\mathfrak S}(\mathcal M')\subseteq Ass_{{_T}\mathfrak S}(\mathcal M)\subseteq Ass_{{_T}\mathfrak S}(\mathcal M')\cup Ass_{{_T}\mathfrak S}(\mathcal M'')
\end{equation*}
(3) If $\mathcal M'\subseteq \mathcal M$ is an essential subobject, then $Ass_{{_T}\mathfrak S}(\mathcal M')=Ass_{{_T}\mathfrak S}(\mathcal M)$.

\smallskip
(4) Every injective in ${_T}\mathfrak S$ can be expressed as a direct sum of injective envelopes of $T$-elementary objects in $\mathfrak S$.

\smallskip
For our purposes in this paper, we will need to refine the result of (4).

\begin{lem}\label{L8.17fx}
Let $\mathfrak S$ be a strongly locally noetherian $k$-linear category and let $T$ be a commutative noetherian $k$-algebra. Let $0\ne \mathcal E\in {_T}\mathfrak S$ be an injective object and $\mathfrak p\in Ass_{{_T}\mathfrak S}(\mathcal E)$ be a prime ideal. Then, there exists an injective $\mathcal E(\mathfrak p)\subseteq \mathcal E$ such that:

\smallskip
(1) $\mathcal E(\mathfrak p)$ is a direct sum of injective envelopes of $T$-elementary objects

\smallskip
(2)  $Ass_{{_T}\mathfrak S}(\mathcal E(\mathfrak p))=\{\mathfrak p\}$ and $\mathcal E(\mathfrak p)$ is maximal with respect to the  collection of subobjects $\mathcal N\subseteq \mathcal E$ such that $Ass_{{_T}\mathfrak S}(\mathcal N)=\{\mathfrak p\}$.
\end{lem}

\begin{proof} We consider families $\{\mathcal E_i\}_{i\in I}$ of subobjects of $\mathcal E$ satisfying the following conditions:

\smallskip
(a) Each $\mathcal E_i$ is the injective envelope of a $T$-elementary object $\mathcal L$ such that $Ann_{{_T}\mathfrak S}(\mathcal L)=\{\mathfrak p\}$. 

\smallskip
(b) The sum $\sum_{i\in I}\mathcal E_i\subseteq \mathcal E$ is direct.

\smallskip
Since ${_T}\mathfrak S$ is a Grothendieck category, the condition (b) is equivalent to the assumption that $0=\mathcal E_i\cap (\sum_{j\in J}\mathcal E_j)$ for every $i\in I$ and every finite subset $J\subseteq I\backslash \{i\}$. By Zorn's lemma, it follows that there exists a maximal such family $\{\mathcal E_i\}_{i\in I_0}$. We now set $\mathcal 
E(\mathfrak p):=\underset{i\in I_0}{\bigoplus}\mathcal E_i$. 
Since ${_T}\mathfrak S$ is locally noetherian, we see that the direct sum $\mathcal 
E(\mathfrak p):=\underset{i\in I_0}{\bigoplus}\mathcal E_i$ is injective. 

\smallskip
We now claim that $Ass_{{_T}\mathfrak S}(\mathcal E(\mathfrak p))=\{\mathfrak p\}$. Indeed, from properties (2) and (3) above, it is clear that $Ass_{{_T}\mathfrak S}\left(\underset{j'\in J'}{\bigoplus}\mathcal E_{j'}\right)=\{\mathfrak p\}$ for any finite direct sum with $J'\subseteq I_0$. Since $T$-elementary objects are finitely generated, it follows that any $T$-elementary object $\mathcal L\subseteq \mathcal E(\mathfrak p)$ is contained in a finite direct sum of objects in $I_0$. Hence, $Ann_{{_T}\mathfrak S}(\mathcal L)=\mathfrak p$, which shows that $Ass_{{_T}\mathfrak S}(\mathcal E(\mathfrak p))=\{\mathfrak p\}$.

\smallskip 
Finally, suppose that we have $\mathcal E(\mathfrak p)\subsetneq\mathcal N\subseteq \mathcal E$ with 
$Ass_{{_T}\mathfrak S}(\mathcal N)=\{\mathfrak p\}$. Since $\mathcal E(\mathfrak p)$ is injective, we can write $\mathcal N=\mathcal E(\mathfrak p)\oplus\mathcal N'$, with $\mathcal N'\ne 0$. Then, $\phi\ne Ass_{{_T}\mathfrak S}(\mathcal N')\subseteq Ass_{{_T}\mathfrak S}(\mathcal N)=\{\mathfrak p\}$, which gives $Ass_{{_T}\mathfrak S}(\mathcal N')=\{\mathfrak p\}$. Hence, there is a $T$-elementary object $\mathcal L'\subseteq \mathcal N'$ satisfying $Ann_{{_T}\mathfrak S}(\mathcal L')=\{\mathfrak p\}$. 

\smallskip
Since $\mathcal E$ is injective, we now consider an injective envelope $\mathcal E(\mathcal L')\subseteq \mathcal E$ of $\mathcal L'$. Since $\mathcal L'\subseteq \mathcal N'$, we must have
$\mathcal E(\mathfrak p)\cap \mathcal L'=0$. Since $\mathcal L'\subseteq \mathcal E(\mathcal L')$ is essential, it now follows that $\mathcal E(\mathfrak p)\cap \mathcal E(\mathcal L')=0$. Then, the family $\{\mathcal E_i\}_{i\in I_0}\cup \{\mathcal E(\mathcal L')\}$ also satisfies the conditions (a) and (b), which contradicts the maximalilty of the family $\{\mathcal E_i\}_{i\in I_0}$. 

\end{proof}

\begin{thm}\label{P8.18xk}  Let $\mathfrak S$ be a strongly locally noetherian $k$-linear category and let $T$ be a commutative noetherian $k$-algebra. Let $\mathcal E\in {_T}\mathfrak S$ be an injective object. For each prime ideal $\mathfrak p\in Ass_{{_T}\mathfrak S}(\mathcal E)$, let $\mathcal E(\mathfrak p)\subseteq \mathcal E$ be an injective that is maximal   with respect to the  collection of subobjects $\mathcal N\subseteq \mathcal E$ such that $Ass_{{_T}\mathfrak S}(\mathcal N)=\{\mathfrak p\}$. Then, we can write
\begin{equation}
\mathcal E\cong \underset{\mathfrak p\in Ass_{{_T}\mathfrak S}(\mathcal E)}{\bigoplus}\textrm{ }\mathcal E(\mathfrak p)
\end{equation}
\end{thm}

\begin{proof}
We will first show that the sum $\underset{\mathfrak p\in Ass_{{_T}\mathfrak S}(\mathcal E)}{\sum}\textrm{ }\mathcal E(\mathfrak p)\subseteq \mathcal E$ is direct. Since ${_T}\mathfrak S$ is a Grothendieck category, it suffices to show that $\mathcal E(\mathfrak p_1)+ ... + \mathcal E(\mathfrak p_k)\subseteq \mathcal E$ is direct for any finite collection $\{\mathfrak p_1,...,
\mathfrak p_k\}$ of prime ideals in $Ass_{{_T}\mathfrak S}(\mathcal E)$. This is obviously true for $k=1$. Proceeding by induction, if $k\geq 2$ and we set $\mathcal N'=\mathcal E(\mathfrak p_k)\cap (\mathcal E(\mathfrak p_1)\oplus ...\oplus\mathcal E(\mathfrak p_{k-1}))$, we have
\begin{equation}\label{819fy}
\begin{array}{c}
Ass_{{_T}\mathfrak S}(\mathcal N')\subseteq Ass_{{_T}\mathfrak S}(\mathcal E(\mathfrak p_k))=\{\mathfrak p_k\} \\
Ass_{{_T}\mathfrak S}(\mathcal N')\subseteq Ass_{{_T}\mathfrak S}(\mathcal E(\mathfrak p_1)\oplus ...\oplus\mathcal E(\mathfrak p_{k-1}))\subseteq \{\mathfrak p_1,...,\mathfrak p_{k-1}\}
\end{array}
\end{equation} From \eqref{819fy}, it follows  that $Ass_{{_T}\mathfrak S}(\mathcal N')=\phi$, which gives $\mathcal N'=0$. Hence, the sum $\mathcal E(\mathfrak p_1)+ ... + \mathcal E(\mathfrak p_k)\subseteq \mathcal E$ is direct. We now set $\mathcal E':=\underset{\mathfrak p\in Ass_{{_T}\mathfrak S}(\mathcal E)}{\bigoplus}\textrm{ }\mathcal E(\mathfrak p)$. Because  ${_T}\mathfrak S$ is locally noetherian,  the direct sum $\mathcal E'$ of injectives is injective. We now write 
$\mathcal E=\mathcal E'\oplus \mathcal E''$. 

\smallskip
Suppose that $\mathcal E''\ne 0$. Then, we can find a $T$-elementary object $\mathcal L\subseteq \mathcal E''$. Since $\mathcal E''$ is injective, we have $\mathcal E(\mathcal L)
\subseteq \mathcal E''$ for an injective envelope $\mathcal E(\mathcal L)$ of $\mathcal L$. But if $Ann_{{_T}\mathfrak S}(\mathcal L)=\{\mathfrak q\}$, this contradicts the maximality of
$\mathcal E(\mathfrak q)$. Hence, $\mathcal E''=0$. 
\end{proof}

We now return to the context of $(A,H)$-Hopf module objects in $\mathfrak S$, along with the assumption that $A$ is left noetherian, $\mathfrak S$ is strongly locally left noetherian and that the injective cogenerator $\hat{\mathcal I}=\underline{Hom}(A,\mathcal I)\otimes H\in {_A}\mathfrak S^H$, which gives $ {_A}\mathfrak S^H=\mathcal F({_A}\mathfrak S^H)$. We take
$\mathcal M\in {_A}\mathfrak S^H$. By Theorem \ref{T8.5gv} and Proposition \ref{P7.11sgs}, we know that
\begin{equation}\label{iso821}
\mathfrak C^H({_A}\mathcal E^H(\mathcal M))= \mathfrak C^H({_B}HOM(A,{_B}\mathcal E(\mathfrak C^H(\mathcal M))))={_B}\mathcal E(\mathfrak C^H(\mathcal M))
\end{equation} Here, ${_A}\mathcal E^H(\mathcal M)$ is the injective envelope of $\mathcal M\in {_A}\mathfrak S^H$ and ${_B}\mathcal E(\mathfrak C^H(\mathcal M))$ is the injective envelope
of $\mathfrak C^H(\mathcal M)\in {_B}\mathfrak S$. Our final aim is to obtain a direct sum decomposition for all the terms appearing in a minimal injective resolution of $\mathcal M\in {_A}\mathfrak S^H$
\begin{equation}\label{res822}
\mathcal M\xrightarrow{\psi^0} {_A}\mathcal E^{H0}(\mathcal M)\xrightarrow{\psi^1} {_A}\mathcal E^{H1}(\mathcal M)\xrightarrow{\psi^2} {_A}\mathcal E^{H2}(\mathcal M)\xrightarrow{\psi^3} \dots
\end{equation} Because \eqref{res822} is a minimal resolution, we note that each ${_A}\mathcal E^{H\bullet }(\mathcal M)/Im(\psi^\bullet)\hookrightarrow {_A}\mathcal E^{H\bullet+1}(\mathcal M)$ is an essential monomorphism.

\begin{lem}\label{L820g}
Let $\mathcal M\in {_A}\mathfrak S^H$. Then, $\mathfrak C^H({_A}\mathcal E^{H\bullet}(\mathcal M))$ is a minimal injective resolution of $\mathfrak C^H(\mathcal M)$ in ${_B}\mathfrak S$, i.e., 
$\mathfrak C^H({_A}\mathcal E^{H\bullet}(\mathcal M))={_B}\mathcal E^\bullet(\mathfrak C^H(\mathcal M))$. 
\end{lem}

\begin{proof}
Since $\mathfrak C^H$ is exact, we know that the complex $\mathfrak C^H(\mathcal M)\longrightarrow \mathfrak C^H({_A}\mathcal E^{H\bullet}(\mathcal M))$ is exact. Since each ${_A}\mathcal E^{H\bullet}(\mathcal M)\in  {_A}\mathfrak S^H$ is injective, it follows by Corollary \ref{C8.75s} that each $\mathfrak C^H({_A}\mathcal E^{H\bullet}(\mathcal M))\in {_B}\mathfrak S$ is an injective object. We are given that each of the following is an essential monomorphism
\begin{equation}\label{ess823}
{_A}\mathcal E^{H\bullet }(\mathcal M)/Im(\psi^\bullet)\hookrightarrow {_A}\mathcal E^{H\bullet+1}(\mathcal M)={_B}HOM(A, \mathfrak C^H({_A}\mathcal E^{H\bullet+1}(\mathcal M)))
\end{equation} where the equality in \eqref{ess823} follows from Theorem \ref{T8.5gv}  and Corollary \ref{C8.75s}. Applying Lemma \ref{L8.16cop} and using the fact that $\mathfrak C^H$ is exact, we now see that each 
\begin{equation}
\mathfrak C^H({_A}\mathcal E^{H\bullet }(\mathcal M))/Im(\mathfrak C^H(\psi^\bullet))\hookrightarrow \mathfrak C^H({_A}\mathcal E^{H\bullet+1}(\mathcal M))
\end{equation}  is an essential monomorphism. This proves the result.
\end{proof}

\begin{Thm}\label{Tfin6}
Suppose that

\smallskip
(1) $A$ is a  left noetherian $k$-algebra and a right $H$-comodule algebra

\smallskip
(2) $B=A^{coH}$ is a noetherian and commutative $k$-algebra

\smallskip
(3) $\mathfrak S$ is strongly locally left noetherian

\smallskip
(4)  $\mathfrak S$ has an injective cogenerator $\mathcal I$ which satisfies 
 $\hat{\mathcal I}=\underline{Hom}(A,\mathcal I)\otimes H\in  \mathcal F({_A}\mathfrak S^H)$. 
 
 \smallskip
 Then for $\mathcal M\in {_A}\mathfrak S^H$, we have a direct sum decomposition
\begin{equation}\label{827qz}
{_A}\mathcal E^{H\bullet}(\mathcal M)= \underset{\mathfrak p\in Ass_{{_B}\mathfrak S}({_B}\mathcal E^\bullet(\mathfrak C^H(\mathcal M)))}{\bigoplus}\textrm{ } {_B}HOM(A,{_B}\mathcal E^\bullet(\mathfrak C^H(\mathcal M))(\mathfrak p))
\end{equation}
\end{Thm}

\begin{proof}
By Lemma \ref{L820g}, we have $\mathfrak C^H({_A}\mathcal E^{H\bullet}(\mathcal M))={_B}\mathcal E^\bullet(\mathfrak C^H(\mathcal M))$. Since $B$ is  a noetherian and commutative $k$-algebra, it now follows from Proposition \ref{P8.18xk} that
\begin{equation}\label{825qz}
\mathfrak C^H({_A}\mathcal E^{H\bullet}(\mathcal M))={_B}\mathcal E^\bullet(\mathfrak C^H(\mathcal M))= \underset{\mathfrak p\in Ass_{{_B}\mathfrak S}({_B}\mathcal E^\bullet(\mathfrak C^H(\mathcal M)))}{\bigoplus}\textrm{ }{_B}\mathcal E^\bullet(\mathfrak C^H(\mathcal M))(\mathfrak p)
\end{equation} Since ${_A}\mathcal E^{H\bullet}(\mathcal M)\in {_A}\mathfrak S^H$ is injective, it follows by Theorem \ref{T8.5gv} and Corollary \ref{C8.75s} that 
\begin{equation}\label{826qz}
{_A}\mathcal E^{H\bullet}(\mathcal M)={_B}HOM(A, \mathfrak C^H({_A}\mathcal E^{H\bullet}(\mathcal M)))={_B}HOM\left(A, \underset{\mathfrak p\in Ass_{{_B}\mathfrak S}({_B}\mathcal E^\bullet(\mathfrak C^H(\mathcal M)))}{\bigoplus}\textrm{ }{_B}\mathcal E^\bullet(\mathfrak C^H(\mathcal M))(\mathfrak p) \right)
\end{equation} Applying Lemma \ref{L8.10db}, we now have the direct sum decomposition in \eqref{827qz}. 
\end{proof}

\end{document}